\documentclass[11pt]{amsart}

\usepackage{graphicx,amssymb,amsmath,xcolor,enumerate}
\usepackage{hyperref}
\hypersetup{
	bookmarks=true,         %show bookmarks bar?
	pdffitwindow=false,     % window fit to page when opened
	pdfstartview={FitH},    % fits the width of the page to the window
	colorlinks=true,      %false: boxed links; true: colored links
	citecolor=blue,
}
\usepackage{dsfont}
\usepackage{bbm}
\usepackage{enumitem}

\usepackage{geometry}
\DeclareSymbolFont{largesymbol}{OMX}{yhex}{m}{n}
\DeclareMathAccent{\Widehat}{\mathord}{largesymbol}{"62}

\newcommand{\R}{\mathbb{R}}

\numberwithin{equation}{section}              % numerazionedelleequazioni
\newtheorem{theorem}{Theorem}[section]

\newtheorem{lemma}{Lemma}[section]
\newtheorem{proposition}{Proposition}[section]
\newtheorem*{proposition*}{Proposition}

\newtheorem*{corollary*}{Corollary}
\newtheorem{definition}{Definition}[section]
\newtheorem*{definitions*}{Definitions}

\newtheorem*{acknowledgements*}{Acknowledgements}

\newtheorem*{conjecture*}{\bf Conjecture}

\newtheorem*{example*}{\bf Example}
\theoremstyle{remark}
\newtheorem{remark}{\bf Remark}[section]

	\date{ }

	\author{Changfeng Gui}
	\address{
		Department of Mathematics, University of Macau, Taipa, Macau and
		Zhuhai UM Science and Technology Research Institute, Hengqin, Guangdong, 519031, China}
	\email{changfenggui@um.edu.mo}
	
	\author{Hao Liu}
	\address{Department of Mathematics, Faculty of Science, University of Macau, Taipa, Macau}
	\email{haoliu@um.edu.mo}
	\author{Chunjing Xie}
	\address{School of Mathematical Sciences, 
		Ministry of Education Key Laboratory of Scientific and Engineering Computing,
		and CMA-Shanghai, Shanghai Jiao Tong University, 800 Dongchuan Road, Shanghai, China}
	\email{cjxie@sjtu.edu.cn}

	\title[Forward DSS solutions to 2D Navier-Stokes equations]{Large forward discretely self-similar solutions to the two-dimensional Navier-Stokes equations with rough data}
	
		
\begin{document}

	\begin{abstract}

        We construct forward discretely self-similar (DSS) solutions  of two-dimensional Navier-Stokes equations
        with arbitrarily large data in $L^2_{loc}(\mathbb{R}^2\setminus\{0\})$.
        For such rough initial data, we first establish  a sharp linear estimate under the discrete scaling, which identifies the \(L^2\)-norm of the initial datum on a fundamental annulus with the  Dirichlet energy of its caloric lift over one period in logarithmic time.
         This, together with a structural cancellation between the linear profile and the nonlinear remainder over one period in logarithmic time,  yields a uniform 
 \(\dot{H}^1\)-bound for the remainder, where  the natural $L^2_{loc}$-regularity of the initial data seems to be sharp.
         Since the two-dimensional Leray operator has no \(L^2\)-coercivity, we recover uniform energy control for the remainder by exploiting  \(L^p\)-estimates ($1<p<2)$, together with a key  bootstrap argument and periodicity.
        We then first  construct smooth DSS solutions  using further weighted estimates and a  fixed-point argument  as long as the initial data belong to $C^2_{loc}(\mathbb{R}^2\setminus\{0\})$. These stronger weighted bounds are not uniform under approximation of rough data. Instead, the passage to \(L^2_{\rm loc}\) data relies only on the unweighted energy estimates. Finally, we use strong time continuity of the limiting periodic profile to recover the uniform spatial tightness needed to verify the initial data.
	\end{abstract}
	
	\maketitle

	\section{Introduction and Main Results}

	\subsection{Introduction}

We consider the two-dimensional incompressible Navier-Stokes equations
\begin{equation}\label{eq:NS}
	\left\{
	\begin{aligned}
		&\frac{\partial u}{\partial t}
		-\Delta u
		+u\cdot\nabla u
		+\nabla \mathfrak{p}=0,\\
		&\nabla\cdot u=0,
	\end{aligned}
	\right.
	\quad
	\text{in }\mathbb R^2\times(0,\infty),
\end{equation}
with initial condition
\begin{equation}\label{eq:initial}
	u(x,0)=u_0(x).
\end{equation}
Here
\(
	u:\mathbb R^2\times[0,\infty)\to\mathbb R^2
\)
denotes the velocity field and
\(
	\mathfrak{p}:\mathbb R^2\times[0,\infty)\to\mathbb R
\)
the pressure.

For sufficiently regular solutions with finite-energy initial data $u_0\in L^2(\mathbb{R}^n)$,
the Navier-Stokes equations formally satisfy the energy identity
\begin{equation}\label{eq:enerest}
	\int_{\mathbb R^n}|u(x,t)|^2\,dx
	+
	2\int_0^t\int_{\mathbb R^n}
	|\nabla u(x,s)|^2\,dx\,ds
	=
	\int_{\mathbb R^n}|u_0(x)|^2\,dx.
\end{equation}
Starting from this basic a priori estimate, Leray \cite{Leray34} and
Hopf \cite{Hopf51} constructed global weak solutions satisfying the corresponding
energy inequality for finite-energy
initial data in dimensions two and three. In two dimensions such weak solutions are globally
regular, whereas in three dimensions only partial regularity is known \cite{CKN};
whether smooth solutions can develop singularities in finite time
remains one of the central open problems in the theory of the
Navier-Stokes equations.

An important distinction between the two- and three-dimensional
problems can already be seen from the scaling of the equations.
Indeed, \eqref{eq:NS}-\eqref{eq:initial} are invariant under the scaling
\begin{equation*}
	\begin{aligned}
		u(x,t)
		&\longmapsto
		u_\lambda(x,t)
		=
		\lambda u(\lambda x,\lambda^2t),\\
		\mathfrak{p}(x,t)
		&\longmapsto
		\mathfrak{p}_\lambda(x,t)
		=
		\lambda^2\mathfrak{p}(\lambda x,\lambda^2t),\\
		u_0(x)
		&\longmapsto
		u_{0,\lambda}(x)
		=
		\lambda u_0(\lambda x).
	\end{aligned}
\end{equation*}
If the left-hand side of \eqref{eq:enerest} is denoted by \(E(u)\),
then
\begin{equation*}
	E(u_\lambda)=\lambda^{2-n}E(u).
\end{equation*}
Thus the energy is scale invariant precisely in dimension two.

In this paper, we consider a class of solutions lying  exactly at the critical level and 
outside the classical finite-energy theory of the two-dimensional Navier-Stokes equations. In view of the scaling, a solution
\((u,\mathfrak{p})\) is called \emph{self-similar} or \emph{scale-invariant} if for all  $\lambda > 1$, 
\begin{equation}\label{cond:ss}
u_{\lambda} (x,t) = u(x,t), \quad \mathfrak{p}_{\lambda}(x,t)=\mathfrak{p}(x,t).
\end{equation}
	Correspondingly, we say that the initial condition $u_0$ is self-similar or scale-invariant
	if it satisfies $\lambda u_0(\lambda x)= u_0(x)$ for all $\lambda > 0$.  This is precisely to say $ u_0$  is $(-1)$-homogeneous. In two dimensions, a $(-1)$-homogeneous vector field is not square-integrable  at the origin and at infinity, hence, has infinite energy.  

On the other hand, in many physical problems, solutions satisfy the scale-invariance \eqref{cond:ss} for a particular $\lambda>1$ rather than for all $\lambda$.
 A solution $(u,\mathfrak{p})$ is called discretely self-similar  (DSS) 
	if  
	\begin{equation}\label{eq:defdss1}
		u_{\lambda} (x,t) = u(x,t), \quad \mathfrak{p} _{\lambda}(x,t)=\mathfrak{p}(x,t)\quad \textrm {for some }  \lambda > 1.
	\end{equation}
	Likewise, an initial datum $u_0$ is called DSS if   
	\begin{equation}\label{eq:defdss}
		u_{0,\lambda}( x) =  u_0(x)
       \quad \textrm{for some }  \lambda > 1.
	\end{equation}
In this case, we say that $u$, respectively
$u_0$, is $\lambda$-DSS. It is obvious that  \eqref{eq:defdss1} and \eqref{eq:defdss} also hold for all scaling factors $\lambda^k$, $k\in \mathbb{Z}$, whenever it is $\lambda$-DSS. 
Discrete scale invariance
appears in nonlinear systems for which invariance persists only along
a discrete sequence of scales, and it is typically reflected by
periodicity in logarithmic variables; see, for instance,
\cite{Sornette98}.  We mention two such examples of discrete self-similarity in  thin films of viscous fluid \cite{dallaston2018discrete} and viscous cavity \cite{fontelos2021discrete}.

In two dimensions, the infinite-energy character of self-similar initial data becomes more clear in the broader DSS setting. We  notice that 
	the value of a $\lambda$-DSS vector field $u_0(x)$ in $\R^2\setminus\{0\}$  is uniquely determined by its values in the annulus 
		$\mathcal{A} = \{x \in \R^2 : 1 \leq |x| <\lambda\}$.
And we have
\begin{equation*}
    \int_{\lambda^{k}\mathcal{A}} |u_0|^2 dx =  \int_{\mathcal{A}}|u_0(\lambda^k y)|^2 \lambda^{2k} dy = \int_{\mathcal{A}}|u_0(y)|^2 dy,
\end{equation*}
where we have used $  \lambda^k u_0(\lambda^k y) = u_0(y)$. Hence,  each  member of the disjoint annulus $\lambda^{k}\mathcal{A}$ ($k\in \mathbb{Z}$)
carries the same amount of kinetic energy. Thus for the discretely self-similar vector field, the natural quantity is not the total energy, which is  infinite, but the kinetic energy contained in the fundamental annulus $\mathcal{A}$. We denote
\[
A:=\|u_0\|_{L^2(\mathcal A)}.
\]

\subsection{Previous results}
For sufficiently small initial data, existence of self-similar and discretely
self-similar solutions follow from the well-posedness theory in
scale-critical spaces; see, for example,
\cite{Barraza96,Cannone96,Giga89,Koch01}. The large-data problem is
substantially different, since the nonlinear term can no longer be
treated perturbatively.

  In three dimensions,   the first   large data existence result of forward self-similar solutions was obtained in the seminal work of Jia and {\v{S}}ver{\'{a}}k \cite{Jia14}.  
Large discretely self-similar solutions
were subsequently constructed by Tsai \cite{Tsai2014} and Bradshaw and Tsai
	\cite{Tsai17, Bradshaw17}. The three-dimensional case was further
extended to rough initial data and half-space, see   \cite{Korobkov2016, Xue15, Chae18AIHP, Bradshaw19apde, Bradshaw18arma, Albritton19arma}.

    The two-dimensional problem has a different critical feature. A
nontrivial scale-invariant velocity has order
\(|x|^{-1}\), and therefore does not have locally finite-energy 
around the origin (see also the discussion above for the discretely self-similar data). 
There is also an exact loss of coercivity in similarity variables.
Indeed, if one works directly with the Leray equations in self-similar variables (see \eqref{eq:Leray} below), 
 self-similar solutions correspond to stationary solutions. In energy estimates for Leray equations,
the zeroth-order term and the  drift cancel in the $L^2$-energy identity, that is
\begin{equation}\label{eq:L2criticalityintro}
	\int_{\mathbb R^2}
	\left(
	-\frac12V-\frac12y\cdot\nabla V
	\right)\cdot V\,dy
	=0.
\end{equation}   
    Consequently, the natural energy method controls only the $\dot{H}^1$ norm, which does not even prevent growth at spatial infinity in two dimensions, in contrast to the three-dimensional case \cite{Korobkov2016,Tsai17, Bradshaw17}.
	Only very recently,
	the existence of large self-similar solutions to the two-dimensional Navier-Stokes equations with locally H\"{o}lder continuous initial data has been proved  in \cite{albritton2026forward} and \cite{gui2026forward}   by different approaches. 
 
DSS solutions are periodic solutions in
the similarity variables. Thus the DSS problem is genuinely
time-dependent  rather than stationary.
     The first objective of this paper is to show
    existence of discretely self-similar solutions to 2D Navier-Stokes equations for arbitrary scaling factor $\lambda$.
    The second main goal of this paper is to extend the existence result to rough initial data in $L^2_{\mathrm{loc}}(\mathbb{R}^2\setminus\{0\})$, or equivalently $L^{2,\infty}(\mathbb{R}^2)$, which is the natural  space for discretely self-similar initial data, and can be  viewed as a good counterpart for  $L^2(\mathbb{R}^2)$ initial data.
   Indeed, we also find that the  $L^2_{\mathrm{loc}}$ regularity is not merely a technical assumption; it is the sharp regularity level at which the linear heat profile has finite Dirichlet energy over one logarithmic period.

		\subsection{Main results}
	Before stating our result, we give the definition of the solution to \eqref{eq:NS}, which is essentially the same as \cite[Definition 8.18]{Tsai18}. 
	\begin{definition}\label{def:energy persol}
		A vector field $u$ defined on $\mathbb{R}^2\times[0,\infty)$ is called an energy perturbed solution with divergence-free initial data $u_0\in L^{2,\infty}(\mathbb{R}^2) $ 
		if it satisfies  \eqref{eq:NS} in the sense of distributions, and \begin{itemize}
			\item[(i)]  $u- e^{t\Delta}u_0 \in L^{\infty}(0,T;L^2(\mathbb{R}^2))\cap L^{2}(\epsilon,T;H^1(\mathbb{R}^2))$ for any $\epsilon, T>0$,
			\item[(ii)]  Initial condition: $\lim_{t\to 0} \|u(t)- e^{t\Delta}u_0 \|_{L^2_{\mathrm{loc}}(\mathbb{R}^2\setminus\{0\})} = 0$,
		\end{itemize}
		where \begin{equation*}
			(e^{t\Delta}u_0) (x) = \int_{\mathbb{R}^2} \frac{1}{4\pi t}e^{-\frac{|x-z|^2}{4 t}}u_0(z)dz.
		\end{equation*}
		is the solution to the heat equation with the same initial data $u_0$.
	\end{definition}
	
	The main results in this paper are as follows.
	\begin{theorem}\label{thm:main}
		Let $u_0\in L^{2}_{loc}(\R^2\setminus\{0\})$ be a divergence-free, $\lambda$-DSS vector field  in $\R^2 \setminus \{0\}$, such that
		 \begin{equation}\label{eq:zeroflux}
			  \int_{ \mathbb{S}^1} u_0\cdot n d\sigma = 0 .
		\end{equation}
		Then	there exists a discretely  self-similar solution $u (x,t)$ to the Cauchy problem \eqref{eq:NS}-\eqref{eq:initial}, 
		which is  an energy perturbed solution in the sense of Definition \ref{def:energy persol}.  Moreover,
		for any $1<p<\infty $, we have
			\begin{equation*}
			\left\| u(x,t)-  e^{t\Delta}u_0\right \|_{L^p(\mathbb{R}^2)}\leq C(A,\lambda,p)t^{-\frac{1}{2}+
			\frac{1}{p}},
		\end{equation*}
where  
		$A=\|u_0\|_{L^{2}(\mathcal{A})}<+\infty$ with  $\mathcal{A} = B_\lambda \setminus B_1=\{x \in \R^2 : 1 \leq |x| < \lambda\}$.
	\end{theorem}
	
	There are a few remarks in order. 
	\begin{remark}
		Theorem \ref{thm:main} indeed implies that $u(\cdot,t)\to u_0$ in $L^p(\mathbb{R}^2)$ for
 $1<p<2$, 	and one has for $t\to 0$ that
			\begin{equation*}
			\left\| u(x,t)-  e^{t\Delta}u_0\right \|_{L^p(\mathbb{R}^2)}\leq C t^{-\frac{1}{2}+\frac{1}{p}} \to 0.
		\end{equation*}
			On the other hand, for $p>2$, 	one has the decay estimates as $t\to \infty$, 
				\begin{equation*}
				\left\| u(x,t)-  e^{t\Delta}u_0\right \|_{L^p(\mathbb{R}^2)}\leq C t^{-\frac{1}{2}+\frac{1}{p}} \to 0.
			\end{equation*}
	\end{remark}
	\begin{remark}
		For a  discretely  self-similar  vector field, $u_0\in L^{2}_{loc}(\mathbb{R}^2\setminus\{0\}) $ is equivalent to $u_0\in L^{2,\infty}(\mathbb{R}^2) $. One can refer to \cite[Lemma 3.1]{Tsai17} for a proof, although the proof there is for $\mathbb{R}^3$, it applies to  any dimensions.
	\end{remark}
	\begin{remark}\label{rem:deltaatorigin}
		Any divergence-free vector field  in $\R^2 \setminus \{0\}$ satisfies that
		\begin{equation*}
			\int_{\partial B_r} u_0\cdot n \, d\sigma = \textrm{constant, ~ for any } r>0.
		\end{equation*}
			The condition \eqref{eq:zeroflux} is equivalent to say that $\nabla \cdot u_0 =0$  on $\mathbb{R}^2$ across the origin, which is necessary for the solution $u$ to be divergence-free on $\mathbb{R}^2$ for all time $t>0$. On the other hand,
the vector field
		\begin{equation*}
			u_0(x) = \frac{1}{2\pi}\frac{x}{|x|^2}
		\end{equation*}
	 satisfies $\nabla \cdot u_0=\delta$ in the sense of distributions on $\mathbb{R}^2$.
	In view of this example, if one do not require the divergence-free condition across the origin for initial data, it is natural to seek self-similar solutions that maintain the  singularity $\frac{x}{|x|^2}$ for all time. However, analyzing such solutions is a  quiet different  problem due to the persistence of the strong singularity.
	\end{remark}

	\subsection{Main difficulties and the key ideas of the proof}\label{sec:Main difficulties}
In this subsection, we describe the key ideas for proving the main result, where the discretely self-similarity and the low-regularity of the initial data brings new difficulties. The proofs
rely on three distinct ingredients: a linear Dirichlet energy estimate 
 for the heat flow under the discrete scaling, a periodic nonlinear energy bootstrap argument, and a
rough-data 
approximation.

\textbf{The  \(L^2\) threshold and the  heat flow for DSS initial data.}
Let \(f\) satisfy
\[
f(\lambda x)=\lambda^{-1}f(x).
\]
A first key point is the two-sided estimate
\begin{equation}\label{eq:linearKeyIntro}
	C(\lambda)^{-1}
	\|f\|_{L^2(\mathcal A)}
	\le
	\|\nabla e^\Delta f\|_{L^2(\mathbb R^2)}
	\le
	C(\lambda)
	\|f\|_{L^2(\mathcal A)}.
\end{equation}
Thus the \(L^2\)-energy of the datum on the fundamental annulus $\mathcal A$ is
equivalent to the  Dirichlet energy of its caloric lift at any positive time.
Writing
\[
f(r,\theta)=r^{-1}g(\log r,\theta),
\]
the function \(g\) is periodic in the logarithmic radial variable and in $\theta$.
A Fourier expansion in \((\log r,\theta)\), combined with the
Mellin transform of Bessel functions, helps us to express the  Dirichlet
energy into a weighted norm on one logarithmic  period.
Comparing this norm  with the \(L^2\)-norm of the initial datum on the fundamental annulus gives
\eqref{eq:linearKeyIntro}. In particular, the  \(L^2\)-regularity
 is precisely the
regularity level at which  the Dirichlet energy of the  heat flow can  be finite.

\textbf{Periodic cancellation, subquadratic coercivity, and
	 bootstrap.}
We now introduce the self-similar variables and the Leray equations. 
Let
\begin{equation}\label{eq:changeofvaribales}
	y=\frac{x}{\sqrt{t}} ~\textrm{and}~ s= \log  t.
\end{equation}
We  define   $U(y,s)$  by
\begin{equation}\label{eq:changeofvareq}
	u(x,t)=\frac{1}{\sqrt{t}}U(y,s) = \frac{1}{\sqrt{t}}U\left(\frac{x}{\sqrt{t}}, \log  t\right).
\end{equation}
It follows that if $u$ solves \eqref{eq:NS}, then $U$ satisfies the time-dependent \emph{Leray equations}:
\begin{equation}\label{eq:Leray}
	\left\{
	\begin{aligned}
		&\frac{\partial U}{\partial s}-\Delta U -\frac{1}{2}U -\frac{1}{2}y\cdot \nabla U+ U\cdot \nabla U +\nabla \Pi=0,\\
		&\nabla\cdot U=0,
	\end{aligned}
	\right. \textrm{ in } \mathbb{R}^2\times (-\infty,\infty).
\end{equation}
The discretely self-similarity of $u$ implies that $U$ is periodic in $s$, i.e.,
$U(y, s+T)=U(y,s)$ for 
$T = 2  \log \lambda$.  
The initial condition \eqref{eq:initial} is transformed to the spatial condition at infinity, see Section \ref{sec:intialdata}. 

We define the profile $ U_0(y,s)$ by
\begin{equation}\label{eq:defofv0}
	U_0(y,s)=  \sqrt{t}(e^{t\Delta}u_0)(x), ~	y=\frac{x}{\sqrt{t}} ~\textrm{and}~ s= \log  t.
\end{equation}
It is easy to see that $U_0(y,s)$ is also periodic in $s$ with period $T = 2 \log \lambda$. 
Furthermore, \eqref{eq:linearKeyIntro} and and the estimates for the heat equation imply
\begin{equation}\label{eq:U0DirIntro}
	\sup_{0\le s\le T}
	\|\nabla U_0(s)\|_{L^2(\mathbb R^2)}
	\le C(A,\lambda).
\end{equation}
Direct calculations show that $U_0$ satisfies
\begin{equation}\label{eq:equforheateq}
	\frac{\partial U_0}{\partial s}	-\Delta  U_0 -\frac{1}{2} U_0 -\frac{1}{2}y\cdot \nabla  U_0 = 0, ~\nabla \cdot U_0=0.
\end{equation}
Let $	V$ be the remainder part, that is
\begin{equation}\label{eq:deforre}
	V(y,s)=U(y,s)- U_0(y,s).
\end{equation}
Using \eqref{eq:Leray} and \eqref{eq:equforheateq}, one sees that $	V$ satisfies
\begin{equation}\label{eq:difference}
	\left\{
	\begin{aligned}
		&\frac{\partial V}{\partial s}-\Delta V -\frac{1}{2}V -\frac{1}{2}y\cdot \nabla V+  U_0\cdot \nabla V + V\cdot \nabla  U_0 + V\cdot \nabla V +\nabla \Pi=- U_0\cdot\nabla  U_0, \\
		&  \nabla \cdot V=0,
	\end{aligned} 
	\right. \textrm{ in } \mathbb{R}^2.
\end{equation}

	The first  structural observation
	is that  using the \emph{periodicity}  of the solutions in $s$, there is a \emph{cancellation relation} between $ U_0$ and $V$,
which leads to a
	crucial identity over one period (see also \eqref{eq:crucialidentity} below)
	\begin{equation*}
		\int_{0}^{T} \int_{\mathbb{R}^2} |\nabla V|^2 \,dy \, ds + 		2\int_{0}^{T} \int_{\mathbb{R}^2} \nabla  U_0\cdot \nabla V \,dy \, ds =0.
	\end{equation*}
    This identity has been used in  \cite{gui2026forward}, where the solutions to the Leray equations are independent of $s$. 
	Applying H\"older's inequality to the above identity together with \eqref{eq:linearKeyIntro} yields
	\begin{equation}\label{eq:periodid}
		\|\nabla V\|_{L^2(0,T;L^{2}(\mathbb{R}^2))}\leq 4\|\nabla U_0\|_{L^2(0,T;L^{2}(\mathbb{R}^2))}\leq C(A,\lambda).
	\end{equation}
    The point here is that, unlike the  self-similar problem, \eqref{eq:periodid} is intrinsically a
period-averaged identity. It provides no uniform control of \(V(s)\) at a fixed
 time. Moreover, this quadratic energy identity cannot
recover the  \(L^2\)-norm, and is insufficient to control the behaviour at spatial infinity of the solutions.

To overcome this difficulty, we
use the multiplier $|V|^{p-2}V$ ($1<p<2$) to test \eqref{eq:difference} so that
	the convection term $-\frac{1}{2}y\cdot \nabla V$ can be utilized to suppress
	the drift  term $-\frac{1}{2}V$: 
	\begin{equation}\label{eq:corce}
		\int_{0}^{T} \int_{\mathbb{R}^2} \left(-\frac{1}{2}V -\frac{1}{2}y\cdot \nabla V\right)\cdot |V|^{p-2}V \,dy \, ds = 		\int_{0}^{T} 
		\int_{\mathbb{R}^2}\left(\frac{1}{p}-\frac{1}{2}\right) |V|^{p} \,dy \, ds. 
	\end{equation}
    However, due to the vector field $|V|^{p-2}V$ ($1<p<2$) is not divergence free, the pressure term will not disappear, and one has to estiamte
    \begin{equation*}
        \int_{0}^{T} \int_{\mathbb{R}^2} (\nabla \Pi) \cdot |V|^{p-2} V \, dy
    \end{equation*}
The difficulty is that a direct estimate does not close with the
space-time bounds available at this stage. Indeed, by expressing the
pressure through the Riesz transforms and using the
Calder\'on-Zygmund estimates, the most challenging term to control is $V\cdot\nabla V$. By the
Gagliardo-Nirenberg inequality, one has
\[
    \|V\nabla V\|_{L^p}
    \leq
    \|\nabla V\|_{L^2}
    \|V\|_{L^{\frac{2p}{2-p}}}
    \leq
    C
    \|\nabla V\|_{L^2}^{1+\frac p2}
    \|V\|_{L^p}^{1-\frac p2}.
\]
Consequently, after pairing with $|V|^{p-2}V$, one has
\begin{equation}\label{eq:pressure-critical-term}
    \int_0^T
    \|V(s)\|_{L^p}^{\frac p2}
    \|\nabla V(s)\|_{L^2}^{1+\frac p2}\,ds .
\end{equation}
At this point we only know
\[
    \int_0^T\|\nabla V(s)\|_{L^2}^2\,ds\leq C,
\]
and hence \eqref{eq:pressure-critical-term} cannot be controlled
directly by the coercive space-time term
$\int_0^T\|V(s)\|_{L^p}^p\,ds$. Indeed, an application of Young's
inequality in $s$ that can absorb the $L^p$-factor of $V$ would require a power of
$\|\nabla V\|_{L^2}$ strictly larger than two, which is not available.

This is precisely why we introduce a bootstrap quantity
\[
    M:=\sup_{0\leq s\leq T}\|V(s)\|_{L^p}.
\]
Using this quantity $M$, we have
\[
   \int_0^T
    \|V(s)\|_{L^p}^{\frac p2}
    \|\nabla V(s)\|_{L^2}^{1+\frac p2}\,ds
    \leq
    C M^{p/2}
    \left(
        \int_0^T\|\nabla V(s)\|_{L^2}^2\,ds
    \right)^{\frac{2+p}{4}}
    \leq C M^{p/2}.
\]
The bootstrap argument  is then closed  as follows: the above first enables us to obtain  a
space-time bound, which provides a bound at a good time slice $s_0$. Using the periodicity and the differential
inequality in $s$,    similar  control can be propagated over one full period. This gives
\begin{equation}\label{eq:bootstrapIntro}
	M^p
	\le
	C(A,\lambda,p)\bigl(1+M^{p/2}\bigr),
\end{equation}
which yields a bound for $M$ since $p/2<p$.
   This  bootstrap  is one of
the key argument needed for the time-dependent  problem.  
   
Combining the resulting
\(L^\infty_sL^p_y\)-bound with
\eqref{eq:periodid} and interpolation gives the full energy
estimate 
	\begin{equation}\label{eq:energyest}
		\sup_{0\le s\le T}\|V(s)\|_{L^2(\mathbb R^2)}
		+
		\|V\|_{L^2(0,T;H^1(\mathbb R^2))}
		\le C(A,\lambda).
	\end{equation}
 We then further establish $L^\infty_{s}H^1_y$-bound of $V$ by studying the 
vorticity equation.
	 One can refer to Section \ref{sec:energy estimates} for  the detailed proof.

\textbf{Construction of the solution.}
	To construct the solution, we try to use the Leray-Schauder degree argument. The key   is to establish enough a priori estimates for the possbile fixed points and compactness of Leray-Schauder map. Since the domain is unbounded, the compactness needs extra control at spatial infinity. One way to do this is to use weighted energy estimates with a proper growing weight so that one can control the tail. However, when using a growing weight, the linear part $U_0$ and its derivative lack a sufficient decay due to the low regularity of the initial data, and  the estimates cannot be closed. Our idea is to consider
	$C^2_{loc}$ initial data first,  by then  $U_0$ and its derivative has stronger pointwise decay. 
    We use this additional information to derive
weighted energy estimates for the remainder, which prevent loss of mass at spatial infinity.
These estimates, together with  bounds for the periodic Stokes operator in proper weighted spaces enables us to
construct a solution by using the Leray–Schauder fixed point theorem.

	Once we construct solutions with  $C^2_{loc}$ initial data, we  approximate the $L^2_{loc}$ initial data by $C^2_{loc}$ ones. 
	The  point here is that the estimate \eqref{eq:energyest} only depends on $L^2_{loc}$ regularity of the initial data. These bounds together with the Aubin–Lions lemma then yield
strong  convergence, which is sufficient to pass to the nonlinear
term to obtain a solution. 

There remains a subtle issue concerning the initial data. In physical
variables, for a compact set
\(K\Subset\mathbb R^2\setminus\{0\}\),
\[
\|u(t)-e^{t\Delta}u_0\|_{L^2(K)}^2
=
\int_{K/\sqrt t}|V(y,\log t)|^2\,dy.
\]
As \(t\downarrow0\), the set \(K/\sqrt t\) escapes to spatial infinity.
Thus recovering the initial condition is equivalent to obtaining
uniform spatial tightness of the periodic profile. Such tightness
cannot be inherited from the  weighted estimates for smooth data, since those
estimates are not uniform along the approximation.

Instead, after passing to the limit, using Lions-Magenes lemma (see \cite[Chapter III, Section 1, Lemma 1.2]{Temam79}), we prove
$
V\in C(\mathbb T_T;L^2(\mathbb R^2))$.
Hence
$\{V(s):0\le s\le T\}$
	is a compact subset of \(L^2(\mathbb R^2)\), and therefore
\begin{equation*}
	\lim_{R\to\infty}
	\sup_{0\le s\le T}
	\int_{|y|>R}|V(y,s)|^2\,dy
	=0.
\end{equation*}
This recovers the initial condition in Definition \ref{def:energy persol}. In this way the 
limit for the rough-data is obtained entirely from estimates stable at the annular
\(L^2\)-level, without needing a uniform weighted bound for the
approximating sequence.

\subsection{Organization of the paper}	
	The rest of the paper is organized as follows. We prove the energy estimates for the remainder term $V(y,s)=U(y,s)- U_0(y,s)$ in Section \ref{sec:energy estimates}.  The weighted energy estimates  and  the existence of the solution by using the Leary-Schauder fixed point argument were established in Section \ref{sec:passage-smooth-to-L2} when the initial data belong to $C^2_{loc}$.  Finally, in Section \ref{sec:roughdata}, we pass to general
\(L^2_{\mathrm{loc}}\) DSS initial data by a structure-preserving smooth approximation and compactness argument, thereby completing the proof of the main theorem, Theorem \ref{thm:main}.

	\section{The  Energy estimates for the  Leray equations}\label{sec:energy estimates}

	Throughout the rest of this paper, we assume that $u_0\in L^{2}_{loc}(\R^2\setminus\{0\})$ and  denote 
	\begin{equation*}
		\mathcal{A}=B_\lambda\setminus B_1=\{1\leq |x|< \lambda\}, \quad	A=\|u_0\|_{L^{2}(\mathcal{A})}.
	\end{equation*}
	The main objective of this section is to prove the following a priori estimates for $V=U-U_0$.
	\begin{theorem}\label{thm:energyestkey}
		Let $u_0\in L^{2}_{loc}(\R^2\setminus\{0\})$ be a divergence-free, $\lambda$-DSS initial data  such that $A=\|u_0\|_{L^{2}(\mathcal{A})}<+\infty$ and
		\begin{equation*}
			\int_{ \mathbb{S}^1} u_0\cdot n d\sigma = 0 .
		\end{equation*}
		Assume $u$  is a discretely  self-similar solution to \eqref{eq:NS}-\eqref{eq:initial} in the sense of Definition \ref{def:energy persol},
		and that 
		$V$  is a corresponding solution to
		\eqref{eq:difference}. Then one has  
		\begin{equation*}
			\sup_{0\leq s\leq T} \|V(\cdot,s)\|_{H^{1}(\mathbb{R}^2)}  \leq C(A,\lambda),
		\end{equation*}
		where $T=2 \log  \lambda$.
		Moreover,   for any positive constant $p \in(1,\infty)$   there exists a constant  $C=C(A,\lambda, p)$ 
		such that 
		\begin{equation*}
					\sup_{0\leq s\leq T}\|V(\cdot,s)\|_{L^{p}(\mathbb{R}^2)} \leq C(A,\lambda,p).	
		\end{equation*}
	\end{theorem}
	
	\subsection{Dirichlet energy estimates}\label{sec:outline}
	In this section, we first prove the Dirichlet energy estimates of $V$. We formulate it as a separate Proposition.
	\begin{proposition}\label{lem:Dirichlet energy estimates}
		Under the same assumptions in Theorem \ref{thm:energyestkey}, one has
		\begin{equation*}
			 \|\nabla V\|_{L^2(0,T;L^{2}(\mathbb{R}^2))}  \leq C(A,\lambda).
		\end{equation*}
	\end{proposition}
	We begin with two key lemmas considering the Dirichlet energy of the solution to the heat equation with discretely self-similar initial data that belongs to $L^2_{loc}(\mathbb{R}^2)$. 
	
		\begin{lemma}\label{lem:L^2bound}
		Let $f: \mathbb{R}^2 \setminus \{0\} \to \mathbb{R}$ satisfy $f(\lambda x) = \lambda^{-1} f(x)$ for a fixed $\lambda > 1$. Let the annulus $\mathcal{A} = B_\lambda \setminus B_1$. If $f \in L^2(\mathcal{A} )$, then  there exists some  constant $C(\lambda)$, such that
		\begin{equation*}
			\frac{1}{C(\lambda)}	\|f\|_{L^2(\mathcal{A})} \leq \|\nabla(e^{\Delta}f)\|_{L^2(\mathbb{R}^2)} \leq C(\lambda) \|f\|_{L^2(\mathcal{A})}.
		\end{equation*}
	\end{lemma}
	
	\begin{remark}
		We note that the regularity assumption in Lemma \ref{lem:L^2bound} is sharp in the sense that 	$\|\nabla(e^{\Delta}f)\|_{L^2(\mathbb{R}^2)}$ may not be finite if regularity of $f$ is lower than $L^2$.
			\end{remark}
	\begin{proof}
		We write $f(x) = \frac{1}{|x|} g(x)$, where $g(\lambda x) = g(x)$. 
		Because $f \in L^2(\mathcal{A} )$, $g$ is square-integrable with respect to the scale-invariant measure $\frac{dr}{r} d\theta$. Moreover,  the function $\tilde{g}(s,\theta): = g(r ,\theta)$ is periodic in $s:= \log  r$ and $\theta$, with periods  $ \log  \lambda$ and $2\pi$, respectively. 
		We can 	then	decompose $	\tilde{g}$ into Fourier series as
		\begin{equation*}
			\tilde{g}(s,\theta) = \sum_{m=-\infty}^\infty \sum_{n=-\infty}^\infty c_{m,n} e^{im\theta} e^{i \frac{2n\pi}{ \log  \lambda} s}.
		\end{equation*}
Hence, one has
		\begin{equation*}
			g(r, \theta) = \sum_{m=-\infty}^\infty \sum_{n=-\infty}^\infty c_{m,n} e^{im\theta} r^{i \frac{2n\pi}{ \log  \lambda}}.
		\end{equation*}
		It follows that
		\begin{equation*}
		f =   \sum_{m=-\infty}^\infty \sum_{n=-\infty}^\infty c_{m,n} r^{-1+i \frac{2n\pi}{ \log  \lambda}} e^{im\theta}.
		\end{equation*}
		Direct computations show that
		\begin{equation}\label{eq:Parseval}
			\begin{aligned}
			\|f\|_{L^2(\mathcal{A})}&=	\int_{1}^{\lambda}\int_{0}^{2\pi}|f(r,\theta)|^2 r drd\theta
				=\int_{1}^{\lambda}\int_{0}^{2\pi}\frac{1}{r}|g(r,\theta)|^2 drd\theta\\
				&=\int_{0}^{ \log \lambda}\int_{0}^{2\pi}|\tilde{g}(s,\theta)|^2 dsd\theta=(2\pi  \log  \lambda) \sum_{m,n=-\infty}^\infty |c_{m,n}|^2.
			\end{aligned}
		\end{equation}
		Denote $f_{m,n } = r^{-1+i \frac{2n\pi}{ \log  \lambda}} e^{im\theta}$.
		We compute the Fourier transform of $f_{m,n}$. Let $\xi$ denote the frequency variable and  $(\rho, \phi)$ be the polar coordinates in frequency space. 
		Direct computations yield that the Fourier transform of $f_{m,n }$ (viewed as tempered distributions) is given by
				\begin{equation*}
						\begin{aligned}
								\hat{f}_{m,n}(\xi) = 	\hat{f}_{m,n}(\rho, \phi) &= \int_0^\infty \int_0^{2\pi}r^{i \frac{2n\pi}{ \log  \lambda}}\frac{1}{r}e^{im\theta}e^{-i\rho r \cos(\theta-\phi)} r \, d\theta \, dr\\
					&=  \int_0^\infty \int_0^{2\pi}  r^{i \frac{2n\pi}{ \log  \lambda}}e^{im\phi}e^{im\alpha} e^{-i\rho r \cos\alpha}  \, d\alpha \, dr\\
								&= 2\pi (-i)^{|m|} e^{im\phi} \int_0^\infty  r^{i \frac{2n\pi}{ \log \lambda}}J_{|m|}(\rho r) \, dr,\\
		&=\rho^{-1-i \frac{2n\pi}{ \log  \lambda}} \left( 2\pi (-i)^{|m|} e^{im\phi} \int_0^\infty s^{i \frac{2n\pi}{ \log  \lambda}} J_{|m|}(s) \, ds \right),
							\end{aligned}
					\end{equation*}
				where $J_{|m|}$ is the Bessel function of the first kind of order $|m|$, and we have used the  identity
				\begin{equation*}
						\int_0^{2\pi} e^{im\alpha} e^{-i\rho r \cos \alpha}  \, d\alpha  = 2\pi (-i)^mJ_m(\rho r) = 2\pi (-i)^{|m|}J_{|m|}(\rho r)  .
					\end{equation*}
	Using the Mellin transform of the Bessel function, one has
		\begin{equation*}
			\int_0^\infty s^{i \frac{2n\pi}{ \log  \lambda}} J_{|m|}(s) \, ds = 2^{i \frac{2n\pi}{ \log  \lambda}} \frac{\Gamma\left(\frac{|m|+1+i \frac{2n\pi}{ \log  \lambda}}{2}\right)}{\Gamma\left(\frac{|m|+1-i \frac{2n\pi}{ \log  \lambda}}{2}\right)}.
		\end{equation*}
		Then the  Fourier transform of $f$ is $\hat{f}(\rho, \phi) = \frac{1}{\rho} h(\rho, \phi)$, where
		\begin{equation}\label{eq:Fouriertran}
			h(\rho, \phi) = \sum_{m=-\infty}^\infty \sum_{n=-\infty}^\infty c_{m,n} M_{m,n} \rho^{-i \frac{2n\pi}{ \log  \lambda}} e^{im\phi},
		\end{equation}
		and 
		\begin{equation*}
			M_{m,n} = 2\pi (-i)^{|m|}2^{i \frac{2n\pi}{ \log  \lambda}} \frac{\Gamma\left(\frac{|m|+1+i \frac{2n\pi}{ \log  \lambda}}{2}\right)}{\Gamma\left(\frac{|m|+1-i \frac{2n\pi}{ \log  \lambda}}{2}\right)}.
		\end{equation*}
		We note that $|M_{m,n}| = 2\pi$.
		We now compute the $L^2(\mathbb{R}^2)$ norm of $\nabla(e^{\Delta}f)$. One has
		\begin{equation*}
			\begin{aligned}
				\|\nabla(e^{\Delta}f)\|_{L^2(\mathbb{R}^2)}^2 
				 & =\frac{1}{(2\pi)^2} \int_{\mathbb{R}^2} |\xi|^2 e^{-2|\xi|^2} |\hat{f}(\xi)|^2 \, d\xi\\
				 &= \frac{1}{(2\pi)^2} \int_0^\infty \int_0^{2\pi} \rho^2 e^{-2\rho^2} \frac{|h(\rho, \phi)|^2}{\rho^2} \rho \, d\phi \, d\rho\\
				 &= \frac{1}{(2\pi)^2} \int_0^\infty \int_0^{2\pi} \rho e^{-2\rho^2} |h(\rho, \phi)|^2  \, d\phi \, d\rho.
			\end{aligned}
		\end{equation*}
	It follows from \eqref{eq:Fouriertran} that for each $\rho$,
		\begin{equation*}
			\int_0^{2\pi} |h(\rho, \phi)|^2 \, d\phi = 2\pi \sum_{m=-\infty}^\infty \left| \sum_{n=-\infty}^\infty c_{m,n} M_{m,n} \rho^{-i \frac{2n\pi}{ \log  \lambda}} \right|^2.
		\end{equation*}
		Let $S_m(\rho) = \sum_{n=-\infty}^\infty c_{m,n} M_{m,n} \rho^{-i \frac{2n\pi}{ \log  \lambda}} $.  It then follows that
		\begin{equation}\label{eq:fourierid}
			\begin{aligned}
			 	\|\nabla(e^{\Delta}f)\|_{L^2(\mathbb{R}^2)}^2 & = \frac{1}{(2\pi)^2}\int_0^\infty \int_0^{2\pi} \rho e^{-2\rho^2} |h(\rho, \phi)|^2  \, d\phi \, d\rho \\& = \frac{1}{2\pi}	\sum_{m=-\infty}^\infty \int_0^\infty \rho e^{-2\rho^2} |S_m(\rho)|^2 \, d\rho \\
			& =\frac{1}{2\pi} \sum_{m=-\infty}^\infty \int_{-\infty}^\infty e^{2\tau} e^{-2e^{2\tau}} |S_m(e^\tau)|^2 \, d\tau.
			\end{aligned}
		\end{equation}
		One  notes that   $S_m(e^\tau)$ is a  Fourier series in $\tau$ with coefficients $c_{m,n} M_{m,n}$, since
		\begin{equation*}
			S_m(e^\tau) =\sum_{n=-\infty}^\infty c_{m,n} M_{m,n} e^{-i \tau\frac{2n\pi}{ \log  \lambda}}.
		\end{equation*}
		Moreover, we have that
		\begin{equation*}
			S_m(e^{\tau+ \log \lambda})  = 			S_m(e^\tau),
		\end{equation*}
		and that
		\begin{equation}\label{eq:pasav}
			\int_{0}^{ \log  \lambda}|S_m(e^{\tau})|^2d\tau =  \log \lambda\sum_{n=-\infty}^\infty |c_{m,n} M_{m,n}|^2 = (2\pi)^2 \log \lambda\sum_{n=-\infty}^\infty |c_{m,n}|^2.
		\end{equation}
Using the periodicity of $S_m$, we then have
 \begin{equation*}
 	\begin{aligned}
 			\int_{-\infty}^\infty  |S_m(e^\tau)|^2 e^{2\tau} e^{-2e^{2\tau}} \, d\tau 
 		&=\sum_{k=-\infty}^{\infty}\int_{k \log  \lambda}^{(k+1) \log  \lambda}\left|	S_m(e^\tau) \right|^2e^{2\tau} e^{-2e^{2\tau}} \, d\tau \\
 		&=\sum_{k=-\infty}^{\infty}\int_{0}^{ \log  \lambda}\left|	S_m(e^\tau) \right|^2e^{2\tau+2k \log  \lambda} e^{-2e^{2\tau+2k \log \lambda}} \, d\tau \\
 		&= \int_{0}^{ \log  \lambda}\left|	S_m(e^\tau) \right|^2\left(\sum_{k=-\infty}^{\infty}e^{2\tau+2k \log  \lambda} e^{-2e^{2\tau+2k \log \lambda}}\right) \, d\tau.
 	\end{aligned}
		\end{equation*}
		If we let $W(\tau) = \sum_{k=-\infty}^{\infty}e^{2\tau+2k \log  \lambda} e^{-2e^{2\tau+2k \log \lambda}}$, then $W$ is periodic with period $ \log  \lambda$ and it is not hard to see that there exists a constant $C(\lambda)>1$ such that
		\begin{equation*}
			\frac{1}{C(\lambda)}\leq W(\tau)\leq C(\lambda).
		\end{equation*}
		Combining the above, one has
		\begin{equation*}
				\frac{1}{C(\lambda)}	\int_{0}^{ \log \lambda}\left|	S_m(e^\tau) \right|^2 d\tau \leq \int_{-\infty}^\infty  |S_m(e^\tau)|^2 e^{2\tau} e^{-2e^{2\tau}} \, d\tau \leq C(\lambda) \int_{0}^{ \log \lambda}\left|	S_m(e^\tau) \right|^2 d\tau
		\end{equation*}
		Summing over  $m$, using \eqref{eq:Parseval}, \eqref{eq:fourierid} and \eqref{eq:pasav} yields
		\begin{equation*}
		\frac{	1}{C(\lambda)} \|f\|_{L^2(\mathcal{A})}^2\leq\|\nabla(e^{\Delta}f)\|_{L^2(\mathbb{R}^2)}^2 \leq  C(\lambda) \|f\|_{L^2(\mathcal{A})}^2.
		\end{equation*}
        This finishes the proof of the lemma.
	\end{proof}

Lemma \ref{lem:L^2bound} directly implies the following estimate.
	\begin{proposition}\label{lem:L2bound on a period}
		Let $f: \mathbb{R}^2 \setminus \{0\} \to \mathbb{R}$ be a discretely self-similar function satisfying $f(\lambda x) = \lambda^{-1} f(x)$ for some $\lambda > 1$. Assume that $f \in L^2(\mathcal{A})$, where $\mathcal{A}=B_\lambda \setminus B_1$. Let $u(x,t) = (e^{t\Delta}f)(x)$ be the corresponding solution to the heat equation. Define the self-similar profile
		\begin{equation*}
			U_0(y,s) = \sqrt{t} u(x,t), \quad \text{where} \quad y = \frac{x}{\sqrt{t}} \quad \text{and} \quad s =  \log  t.
		\end{equation*}
		Then   one has
		\begin{equation*}
	 \frac{1}{C(\lambda)}	\|f\|_{L^2(\mathcal{A})} \leq \sup_{0\leq s\leq T}	\|\nabla_y U_0(\cdot,s)\|_{ L^2( \mathbb{R}^2)} \leq  C(\lambda) \|f\|_{L^2(\mathcal{A})},
		\end{equation*}
		for some constant $C(\lambda)$, where $T=2 \log \lambda$.
	\end{proposition}
	
	\begin{proof}
	Since the initial datum satisfies $f(\lambda x) = \lambda^{-1} f(x)$, the uniqueness of the heat equation implies that $u(x,t)$ satisfies 
		$	u(\lambda x, \lambda^2 t) = \lambda^{-1} u(x,t)$.
		Consequently, the profile $U_0(y,s)$ is periodic in  $s$ with period $T = 2 \log \lambda$. 
		
	Using the chain rule, we compute that
		\begin{equation*}
			\nabla_y U_0(y,s) = \sqrt{t} \nabla_y \big[ u(y\sqrt{t}, t) \big]  = t \nabla_x u(x,t).
		\end{equation*}
		It follows that
		\begin{align*}
		\int_{\mathbb{R}^2} |\nabla_y U_0(y,s)|^2 \, dy  =  \int_{\mathbb{R}^2} \big( t^2 |\nabla_x u(x,t)|^2 \big) (t^{-1} dx) 
			= t\|\nabla_x u(\cdot, t)\|_{L^2(\mathbb{R}^2)}^2.
		\end{align*}
Lemma \ref{lem:L^2bound}  implies that there exists a constant $C =C(\lambda)$ such that
		\begin{equation*}
			\frac{1}{C(\lambda)}	\|f\|_{L^2(\mathcal{A})}\leq \|\nabla_x u(\cdot, 1)\|_{L^2(\mathbb{R}^2)} \leq  C(\lambda)\|f\|_{L^2(\mathcal{A})}.
		\end{equation*}
	Then due to the heat semigroup is a contraction map on $L^2$, for all $t \geq 1$, we have $\|\nabla_x u(\cdot, t)\|_{L^2(\mathbb{R}^2)} \leq \|\nabla_x u(\cdot, 1)\|_{L^2(\mathbb{R}^2)}$.
		Applying this fact gives
		\begin{equation*}
	 \|\nabla_x u(\cdot, t)\|_{L^2(\mathbb{R}^2)}\leq \|\nabla_x u(\cdot, 1)\|_{L^2(\mathbb{R}^2)}\leq  C(\lambda)\|f\|_{L^2(\mathcal{A})}.
		\end{equation*}
		Hence,
		\begin{align*}
		\sup_{0\leq s\leq T}	 \|\nabla_y U_0(\cdot,s)\|_{L^2(\mathbb{R}^2)}  = \sup_{1\leq t\leq \lambda^2} \sqrt{t}\|\nabla_x u(\cdot, t)\|_{L^2(\mathbb{R}^2)}\leq \lambda C(\lambda)\|f\|_{L^2(\mathcal{A})}.
		\end{align*}
        On the other hand,
        \begin{align*}
	 \|\nabla_y U_0(\cdot,0)\|_{L^2(\mathbb{R}^2)} 
        =  \|\nabla_x u(\cdot, 1)\|_{L^2(\mathbb{R}^2)}\geq \frac{1}{C(\lambda)}	\|f\|_{L^2(\mathcal{A})}.
		\end{align*}
		This completes the proof.
	\end{proof}

	We are now ready to prove Proposition \ref{lem:Dirichlet energy estimates}, and one can see from the proof that the cancellation relation between the linear part $U_0$ and the perturbation part $V$ under the discretely self-similar condition plays a crucial role.
	\begin{proof}[Proof of Proposition \ref{lem:Dirichlet energy estimates}]
 Testing \eqref{eq:difference} with $V$,   using the periodicity of  $V$, after using integrating by parts, one has the energy identity
	\begin{equation}\label{eq:energyesti}
0=	\int_{0}^{T} \int_{\mathbb{R}^2} |\nabla V|^2 + ( V\cdot \nabla  U_0)\cdot V 	+	\int_{0}^{T} \int_{\mathbb{R}^2} ( U_0\cdot\nabla  U_0)\cdot  	V.
	\end{equation}

Next, 	multiplying \eqref{eq:difference} by $U_0$ and integrating on both sides yield
	\begin{equation}
		\begin{aligned}
			0=& \int_{0}^{T}  \int_{\mathbb{R}^2} \frac{\partial V}{\partial s}\cdot U_0 \,dy\,ds +	\int_{0}^{T}  \int_{\mathbb{R}^2} \nabla V: \nabla  U_0   - \frac{1}{2}  U_0\cdot V-\frac{1}{2}(y\cdot \nabla V)\cdot   U_0 \,dy\,ds  \\
			&~ +
			\int_{0}^{T} \int_{\mathbb{R}^2}  ( U_0\cdot \nabla V)\cdot  U_0 +  (V\cdot\nabla V)\cdot   U_0 dy\\
			=& 	\int_{0}^{T}  \int_{\mathbb{R}^2} \frac{\partial V}{\partial s}\cdot U_0 \,dy\,ds+ \int_{0}^{T}  \int_{\mathbb{R}^2} \nabla V: \nabla  U_0 \,dy\,ds    - \frac{1}{2}  U_0\cdot V-\frac{1}{2}(y\cdot \nabla V)\cdot   U_0 \,dy\,ds  \\
			&~ -
			\int_{0}^{T} \int_{\mathbb{R}^2}  ( U_0\cdot \nabla U_0)\cdot  V +  (V\cdot\nabla U_0)\cdot   V d y\\
		\end{aligned}
	\end{equation}	
	Finally, we test \eqref{eq:equforheateq} with $V$ to have
	\begin{equation}\label{eq:cancel}
		\begin{aligned}
			0=&	 \int_{0}^{T}  \int_{\mathbb{R}^2} \frac{\partial U_0}{\partial s}\cdot V \,dy\,ds+ \int_{0}^{T}  \int_{\mathbb{R}^2}  \left(\nabla V:\nabla U_0  -  \frac{1}{2} U_0\cdot V - \frac{1}{2}(y\cdot \nabla U_0)\cdot V \right)\,dy\,ds  \\
			=&\int_{0}^{T}  \int_{\mathbb{R}^2} \frac{\partial U_0}{\partial s}\cdot V \,dy\,ds + \int_{0}^{T}  \int_{\mathbb{R}^2} \left(\nabla V:\nabla U_0  +  \frac{1}{2} U_0\cdot V +
			\frac{1}{2}(y\cdot \nabla V)\cdot  U_0 \right) \,dy\,ds .
		\end{aligned}
	\end{equation}
 Summing \eqref{eq:energyesti}-\eqref{eq:cancel}, thanks to the periodicity of $U_0$ and $V$, one obtains
	\begin{equation}\label{eq:crucialidentity}
		\int_{0}^{T} \int_{\mathbb{R}^2} |\nabla V|^2 \,dy\,ds  + 2 	\int_{0}^{T}\int_{\mathbb{R}^2} \nabla  U_0\cdot \nabla V \,dy\,ds  =0.
	\end{equation}
	This, together with the H\"older inequality, leads to 
	\begin{equation*}\label{eq:dirichletener}
		\int_{0}^{T} \int_{\mathbb{R}^2} |\nabla V|^2 \,dy\,ds  \leq  4 	\int_{0}^{T}\int_{\mathbb{R}^2}|\nabla  U_0|^2 \,dy\,ds \leq C(A,\lambda),
	\end{equation*}
	where we have used Lemma \ref{lem:L2bound on a period}.
	This finishes the  proof.
\end{proof}

	\subsection{Proof of  Theorem \ref{thm:energyestkey}}

	In this section,
	 we return to  \eqref{eq:difference} and  try to utilize the good term $-\frac{1}{2}y\cdot \nabla V$ as much as possible. The main idea is to use the multiplier $|V|^{p-2}V$ for some $1<p<2$.
	If we formally test \eqref{eq:difference} with $|V|^{p-2}V$, we have 
	\begin{equation}\label{eq:formalest}
		\begin{aligned}
			& 	\int_{0}^{T} \int_{\mathbb{R}^2} (p-1)|\nabla V|^2 |V|^{p-2}  + \left(\frac{1}{p}-\frac{1}{2}\right) |V|^{p} \\
			\leq& -  	\int_{0}^{T} \int_{\mathbb{R}^2}  (V\cdot \nabla  U_0)\cdot |V|^{p-2}V + \nabla q\cdot |V|^{p-2}V + ( U_0\cdot \nabla  U_0)\cdot |V|^{p-2} V.
		\end{aligned}
	\end{equation}
	The left hand side contains an additional term $	\int_{0}^{T} \int_{\mathbb{R}^2} \left(\frac{1}{p}-\frac{1}{2}\right) |V|^{p} $, which 
	is the key for us to close the estimates of
	$\|V\|_{L^{p}}$ ($1<p<2$). 
	However, one still needs to control $\int_{0}^{T} \int_{\mathbb{R}^2}  (V\cdot \nabla  U_0)\cdot |V|^{p-2}V$  and $	\int_{0}^{T} \int_{\mathbb{R}^2}  \nabla q\cdot |V|^{p-2}V$. The control of $\int_{0}^{T} \int_{\mathbb{R}^2}  \nabla q\cdot |V|^{p-2}V$ needs a bootstrap argument as explained in Section \ref{sec:Main difficulties}.
 On the other hand, we have
    \begin{equation*}
        \int_{0}^{T} \int_{\mathbb{R}^2}  (V\cdot \nabla  U_0)\cdot |V|^{p-2}V\sim \|\nabla U_0\|_{L^\infty(\mathbb{R}^2)} \int_{0}^{T} \int_{\mathbb{R}^2} |V|^{p}.
    \end{equation*}
    Hence, we need extra smallness to handle this term.
	The idea here is to truncate  
  $U_0$  inside a large ball to ensure it is small enough in $W^{1,\infty}(\mathbb{R}^2)$ globally but still has the same boundary value at infinity.
 Then the above term can be absorbed by the left-hand side of \eqref{eq:formalest}. Finally, once the estimate for $\|V\|_{L^{p}}$  is established, the Gagliardo–Nirenberg interpolation inequality combined with \eqref{eq:dirichletener} yields the desired $L^2$-boundedness of $V$.

We now  construct a modified profile of $U_0$ that is globally small.
	To do so,  let $\eta(x)=\eta(|x|)$ be a smooth non-negative cut-off function
	satisfying 
	\begin{equation*}
		0\leq \eta (|x|) \leq 1, ~	\eta (|x|) =1 \textrm{ for } |x|\geq 2 \textrm{ and } \eta (|x|) =0  \textrm{ for }  |x|\leq 1.
	\end{equation*}
	We also denote for $R_0>1$ that
	\begin{equation}\label{eq:resacledversion}
		\eta_{R_0}(x)=\eta\left(\frac{x}{{R_0}}\right),
	\end{equation}
	then there exists some universal constant $C$ that
	\begin{equation*}
		|\nabla^k \eta_{R_0}|\leq \frac{C_k}{R_0^k}, ~ \forall k\in \mathbb{N}.
	\end{equation*}
	We define a modified version of $U_0$ by 
	\begin{equation}\label{eq:defoftildev}
		U_1= \eta_{R_0}	 U_0 +w, 
	\end{equation}
	where  $w$ solves 
	\begin{equation*}
		\nabla \cdot w = -U_0\cdot \nabla \eta_{R_0}, 
	\end{equation*}
	and hence  $\nabla \cdot U_1=0$.
    Such $w$ can be chosen as
	\begin{equation}\label{eq:corr}
		\begin{aligned}
			w(y,s) &=  - \frac{1}{2\pi}\nabla_y \int_{\mathbb{R}^2} (\ln |y-z|)U_0(z,s)\cdot \nabla  \eta_{R_0}(z)dz
			= -\frac{1}{2\pi} \int_{\mathbb{R}^2} \frac{y-z}{|y-z|^2}U_0 \cdot \nabla  \eta_{R_0}(z)dz.
		\end{aligned}
	\end{equation}

	We then have the following desired estimates for $U_1$,  which is similar to \cite[Lemma 2.5]{Tsai17} and \cite[Lemma 2.2]{gui2026forward}. The proof relies on Lemma \ref{lem:Linftyestforheat} below and 
    follows from similar arguments of \cite[Lemma 2.2]{gui2026forward} with slight modifications to the DSS and low regularity setting, we give the proof in Appendix \ref{sec:App} for reader's convenience.
	\begin{lemma}[Estimates for $  U_1$]\label{lem:Estimates for modified linear part}
		Let $U_0$ be defined in \eqref{eq:defofv0}.
		The vector field $U_1$ defined in \eqref{eq:defoftildev} is smooth, divergence-free,
	and we have
	\begin{equation}\label{eq:estnalbav1}
			\sup_{0\leq s\leq T} \|\nabla U_1\|_{L^2(\mathbb{R}^2)} \leq C(A,\lambda ,R_0).
	\end{equation}
		\begin{equation}\label{eq:estv1}
		\sum_{k=0}^{2}\sup_{0\leq s\leq T}\left\|\nabla^k U_1\right\|_{L^\theta(\mathbb{R}^2)}  \leq C(A,\lambda ,\theta,R_0), ~ \textrm{ for any  } \theta\in(2,\infty].
		\end{equation}
		We have the following decay estimates for $w$ and its time derivative:
		\begin{equation}\label{eq:estw}
			\sup_{0\leq s\leq T} \left| \nabla^k w(y,s)\right|
			\leq
			\frac{C(A,\lambda,k,R_0)}{1+|y|^{k+2}},
			\quad k\in\mathbb N,
		\end{equation}
		and
		\begin{equation}\label{eq:estsw}
			\sup_{0\leq s\leq T} \left| \nabla^k \partial_s w(y,s)\right|
			\leq
			\frac{C(A,\lambda,k,R_0)}{1+|y|^{k+2}},
			\quad k\in\mathbb N.
		\end{equation}
		Finally, for any $\epsilon>0$, we can choose $R_0=R_0(A,\lambda, \epsilon)$ in \eqref{eq:resacledversion} large enough so that 
		\begin{equation}\label{eq:smallness}
			\sup_{0\leq s\leq T} \|  U_1\|_{L^{\infty}(\mathbb{R}^2)}\leq \epsilon, ~ \sup_{0\leq s\leq T} \|\nabla   U_1\|_{L^{\infty}(\mathbb{R}^2)} \leq \epsilon.
		\end{equation}
	\end{lemma}
    The following elementary estimates for the solution to the heat equation with $L^{2}_{loc}$ initial data are needed. 
	\begin{lemma}\label{lem:Linftyestforheat}
		Suppose $u_0\in L^{2}_{loc}(\mathbb{R}^2\setminus\{0\})$ and is $\lambda$-DSS for some $\lambda > 1$ and let $U_0$ be defined 
		 in \eqref{eq:defofv0}. Then, for any $\theta\in(2,\infty]$, we have for any $k\in \mathbb{N}$ that
		\begin{equation*}
		\sup_{0\leq s\leq T}\|\nabla^k U_0\|_{L^\theta(\mathbb{R}^2)} \leq C(A,\lambda ,\theta,k).
		\end{equation*}
		Moreover,
			\begin{equation*}
			\sup_{0\leq s\leq T}\|U_0\|_{L^\theta(|y|>R)} + \sup_{0\leq s\leq T}\|\nabla U_0\|_{L^\theta(|y|>R)}  + \sup_{0\leq s\leq T}\|\nabla^2 U_0\|_{L^\theta(|y|>R)} \leq \Theta(R),
		\end{equation*}
		where $\Theta>0$ depends on $\theta$ and $A$ but satisfies
		$ \Theta(R)\to 0$ as $R\to\infty$.
	\end{lemma}
	\begin{proof}
		 One can see \cite[Lemma 3.2]{Tsai17} for a proof.
	\end{proof}

	With this new profile $U_1= \eta_{R_0}	 U_0 +w$, we define the modified version of $V$ as
	\begin{equation}\label{eq:modre}
		V_1=U- U_1,
	\end{equation}
	where $U$ is the solution to \eqref{eq:Leray}.
	It is straightforward to check that $V_1$ satisfies
	\begin{equation}\label{eq:eqfordiff}
		\left\{
		\begin{aligned}
			&\frac{\partial V_1}{\partial s}-\Delta V_1 -\frac{1}{2}V_1 -\frac{1}{2}y\cdot \nabla V_1+  (U_1+V_1)\cdot \nabla V_1 + V_1\cdot \nabla  U_1 +\nabla \Pi=- U_1\cdot\nabla  U_1 - F(U_0,w),\\
			&  \nabla \cdot V_1=0,		
		\end{aligned}		
		\right.
	\end{equation} 
	where using \eqref{eq:equforheateq}, we have that
	\begin{equation}\label{eq:eqformodifv}
		\begin{aligned}
			F(U_0,w) & =   \frac{\partial U_1}{\partial s}-\Delta  U_1 -\frac{1}{2} U_1 -\frac{1}{2}y\cdot \nabla  U_1 \\
			&	=  		\frac{	\partial (\eta_{R_0}U_0+w)}{\partial s}  -\Delta (\eta_{R_0}	U_0 +w )- \frac{1}{2}(\eta_{R_0}	U_0 +w ) -\frac{1}{2}y\cdot \nabla (\eta_{R_0}	U_0 +w )\\
			&= 	- 2 \nabla \eta_{R_0}\nabla 	U_0- \Delta \eta_{R_0}	U_0   -\frac{1}{2}y\cdot \nabla \eta_{R_0}	 U_0 + \frac{\partial w}{\partial s} -\Delta w  -\frac{1}{2} w -\frac{1}{2}y\cdot \nabla  w.
		\end{aligned}
	\end{equation}
To show Theorem \ref{thm:energyestkey}, it is indeed sufficient to show the corresponding estimates for $V_1$. 
	\begin{proposition}\label{thm:energyest}
		Under the same assumptions of	Theorem \ref{thm:energyestkey}, for any $p\in(1,2)$, we can choose
		$R_0=R_0(A,\lambda, p)$ in   \eqref{eq:defoftildev} large enough, such that 
		the following a priori estimates   for $V_1$ defined in \eqref{eq:modre} hold,
		\begin{equation*}
			\sup_{0\leq s\leq T}\|V_1(\cdot, s)\|_{L^p(\mathbb{R}^2)} \leq C(A,\lambda, p),
		\end{equation*}
		\begin{equation*}
			\sup_{0\leq s\leq T}\|V_1(\cdot, s)\|_{L^2(\mathbb{R}^2)}^2 +	\int_{0}^{T} \int_{\mathbb{R}^2} |\nabla V_1|^2 \,dy \,ds\leq C(A,\lambda, p).
		\end{equation*}
	\end{proposition}
	\begin{remark}
We should mention that  the profile $V_1$ in Theorem \ref{thm:energyest} depend on the cut-off parameter $R_0$, and hence  depend on $p$. But we usually omit this dependence on $p$ since it shall not cause any ambiguity.
	\end{remark}

	\begin{proof}[Proof of Proposition \ref{thm:energyest}]
		Throughout this proof, we let  $C(A,\lambda,p, R_0)$ to denote a generic constant depending on $A,\lambda$, $p$ and $R_0$, which may vary from line to line; the same convention applies to $C(A,\lambda,R_0)$ and $C(A,\lambda,p)$. Moreover, we  use $C_p$ to denote a generic constant depending only on $p$.

	Recall \eqref{eq:defoftildev} and \eqref{eq:modre}, one then has
	\begin{equation*}
	V - V_1 =U_1 -U_0 = (\eta_{R_0}-1)U_0 +w.
\end{equation*}
It follows from Proposition \ref{lem:Dirichlet energy estimates} and Lemmas \ref{lem:Linftyestforheat} and \ref{lem:Estimates for modified linear part} that
	\begin{equation}\label{eq:Dirichletestformodifiedprofile}
		\begin{aligned}
\|\nabla V_1\|_{L^2(0,T;L^{2}(\mathbb{R}^2))}  & \leq \|\nabla V\|_{L^2(0,T;L^{2}(\mathbb{R}^2))}     + \|\nabla ((\eta_{R_0}-1)U_0) + \nabla w \|_{L^2(0,T;L^{2}(\mathbb{R}^2))}   \\
				&\leq C(A,\lambda,R_0).
		\end{aligned}
	\end{equation}
	
		 Multiplying the equation of $V_1$ \eqref{eq:eqfordiff} by $|V_1|^{p-2}V_1$ ($p\in(1,2)$),  integrating  both sides on $\mathbb{R}^2$, it yields that
		\begin{equation}\label{eq:muptiplier1}
			\begin{aligned}
				&  \int_{\mathbb{R}^2} \left(\frac{\partial V_1}{\partial s}	-\Delta V_1 -\frac{1}{2}V_1 -\frac{1}{2}y\cdot \nabla V_1+  (U_1+V_1)\cdot \nabla V_1   \right) \cdot |V_1|^{p-2}V_1\\
				= &-\int_{\mathbb{R}^2}  \left( U_1\cdot\nabla  U_1 + F(U_0,w) + V_1\cdot \nabla  U_1 +\nabla \Pi\right)\cdot  |V_1|^{p-2}V_1 
			\end{aligned}
		\end{equation}

		\textbf{Estimates for the left-hand side of \eqref{eq:muptiplier1}.}
		We estimate the  left-hand side of \eqref{eq:muptiplier1} term by term.
		For the first term,  one has
		\begin{equation*}
 \int_{\mathbb{R}^2} \frac{\partial V_1}{\partial s}\cdot |V_1|^{p-2}V_1 \, dy=\frac{1}{p} \frac{d }{d s} \int_{\mathbb{R}^2} |V_1|^p \, dy.
		\end{equation*}
		For the Laplacian term, we have 
		\begin{equation*}
			\begin{aligned}
				& \int_{\mathbb{R}^2}	-\Delta V_1  \cdot |V_1|^{p-2}V_1 = \int_{0}^{T} \int_{\mathbb{R}^2}	-\Delta V_{1,i}   |V_1|^{p-2}V_{1,i}\\
				=& \int_{\mathbb{R}^2} |\nabla V_1|^2 |V_1|^{p-2} +   \nabla V_{1,i}\cdot \nabla |V_1|^{p-2}V_{1,i}\\
				=& \int_{\mathbb{R}^2} |\nabla V_1|^2 |V_1|^{p-2} +  \frac{1}{2}\nabla |V_1|^{2} \cdot \nabla |V_1|^{p-2}\\
				\geq & \int_{\mathbb{R}^2} |\nabla V_1|^2 |V_1|^{p-2} + (p-2) |V_1|^{p-2} |\nabla V_1|^2= \int_{\mathbb{R}^2} (p-1) |\nabla V_1|^2 |V_1|^{p-2},
			\end{aligned}
		\end{equation*}
		where  the condition $p<2$ guarantees  the second-to-last inequality holds.
		In above, $V_{1,i}$ is the $i$-th component of $V_1$. Next, we have 
		\begin{equation*}\label{eq:est11}
			\begin{aligned}
				\int_{\mathbb{R}^2}	 -\frac{1}{2}y\cdot \nabla V_1\cdot |V_1|^{p-2}V_1\, dy
				&= \int_{\mathbb{R}^2}-\frac{1}{4}y\cdot \nabla |V_1|^2 (|V_1|^{2})^{\frac{p}{2}-1}\, dy\\
				&= \int_{\mathbb{R}^2}-\frac{1}{4}y\cdot \nabla  (|V_1|^{2})^{\frac{p}{2}}\frac{2}{p}\, dy\\
				&=\frac{1}{p}  \int_{\mathbb{R}^2}  |V_1|^{p}\, dy\, ds.\\
			\end{aligned}
		\end{equation*}
	 Integration by parts yields that
		\begin{equation*}\
			\begin{aligned}
			\int_{\mathbb{R}^2}  (U_1+V_1)\cdot \nabla V_1\cdot |V_1|^{p-2}V_1\, dy
				&= \frac{1}{2} \int_{\mathbb{R}^2}	 (U_1+V_1)\cdot \nabla(|V_1|^{2})^{\frac{p}{2}}\frac{2}{p}\, dy\, ds=0.
			\end{aligned}
		\end{equation*}
		Combining all the above estimates gives
		\begin{equation}\label{eq:est1}
			\begin{aligned}
				& \int_{\mathbb{R}^2}	\left(\frac{\partial V_1}{\partial s}-\Delta V_1 -\frac{1}{2}V_1 -\frac{1}{2}y\cdot \nabla V_1+  U_1\cdot \nabla V_1+V_1\cdot \nabla V_1 \right) \cdot |V_1|^{p-2}V_1 \, dy\\
				\geq& \frac{1}{p} \frac{d }{d s} \int_{\mathbb{R}^2} |V_1|^p \, dy +  \int_{\mathbb{R}^2} \left[(p-1)|\nabla V_1|^2 |V_1|^{p-2} + \left(\frac{1}{p}-\frac{1}{2}\right) |V_1|^{p} \right]\, dy.
			\end{aligned}
		\end{equation}

		\textbf{Estimates for the right-hand side of \eqref{eq:muptiplier1}.}
		We now turn to estimate  the right-hand side of \eqref{eq:muptiplier1}.
		We  first note that for any $1<p\leq 2$,
		\begin{equation}\label{eq:forcetermest}
			\begin{aligned}
						\sup_{0\leq s\leq T}\|F(U_0,w)	\|_{L^{p}}\leq C(A,\lambda, p,R_0).
			\end{aligned}
		\end{equation}
		Indeed,  using \eqref{eq:eqformodifv} and  Lemmas \ref{lem:Linftyestforheat} and \ref{lem:Estimates for modified linear part} , for any $1<p\leq 2$,
%		\begin{equation*}
%			\int_{0}^{T} \int_{0}^{T} \int_{\mathbb{R}^2} |\Delta w|^{p} \\
%			\leq C(A,\lambda, p),
%		\end{equation*}
%		\begin{equation*}
%			\left|2 \nabla \eta_{R_0}\nabla 	U_0+\Delta \eta_{R_0}	U_0   +\frac{1}{2}y\cdot \nabla \eta_{R_0}	 U_0  + \frac{1}{2} w + \frac{1}{2}y\cdot \nabla  w\right|\leq \frac{C(A,\lambda, p)}{1+|y|^2}.
%		\end{equation*}
 we have  that
		\begin{equation*}\label{eq:forceterms}
			\begin{aligned}
 \int_{\mathbb{R}^2}	|F(U_0,w)|^{p}
				\leq & C \int_{\mathbb{R}^2} \left|2 \nabla \eta_{R_0}\nabla 	U_0+\Delta \eta_{R_0}	U_0   +\frac{1}{2}y\cdot \nabla \eta_{R_0}	 U_0  +\frac{\partial w}{\partial s}-\frac{1}{2} w + \frac{1}{2}y\cdot \nabla  w +\Delta w \right|^{p} 
				\\
				\leq &
				C(A,\lambda, p,R_0).
			\end{aligned}
		\end{equation*}	
        Furthermore, one has from \eqref{eq:estnalbav1}, \eqref{eq:estv1} and $\frac{2p}{2-p}>2$ (since $1<p<2$) that
		\begin{equation}\label{eq:forceest}
			\sup_{0\leq s\leq T}	\| U_1\|_{L^{\frac{2p}{2-p}}} \leq C(A,\lambda, p, R_0),  ~ \sup_{0\leq s\leq T}\|U_1 \nabla U_1\|_{L^p} \leq  \sup_{0\leq s\leq T} \|U_1\|_{L^{\frac{2p}{2-p}}} \|\nabla U_1\|_{L^2}\leq C(A,\lambda, p,R_0).
		\end{equation}
		Now we can conclude by using the above and \eqref{eq:forcetermest} that
		\begin{equation}\label{eq:force}
			\begin{aligned}
				&\quad	\left|\int_{\mathbb{R}^2}	( U_1\cdot\nabla  U_1 - F(U_0,w))\cdot |V_1|^{p-2}V_1 \right|\\
				&\leq  \int_{\mathbb{R}^2}	 C| U_1\cdot\nabla  U_1|^{p} + C | F(U_0,w)|^{p} + \frac{1}{10}\left(\frac{1}{p}-\frac{1}{2}\right)|V_1|^{p}  \\
				& \leq C(A,\lambda, p,R_0)+  \frac{1}{10}\left(\frac{1}{p}-\frac{1}{2}\right)\int_{\mathbb{R}^2} |V_1|^{p}.
			\end{aligned}
		\end{equation}
			For the second-to-last term on the right-hand side of \eqref{eq:muptiplier1}, using Lemma \ref{lem:Estimates for modified linear part}, one has for $R_0=R_0(A,\lambda,\epsilon)$ that
		\begin{equation}\label{eq:potentialterm}
			\begin{aligned}
				\left| \int_{\mathbb{R}^2}	(V_1\cdot\nabla  U_1) \cdot |V_1|^{p-2}V_1\right| &\leq \sup_{0\leq s\leq T} \|\nabla  U_1\|_{L^{\infty}}  \int_{\mathbb{R}^2} |V_1|^{p}\,dy\\
				&\leq \epsilon \int_{\mathbb{R}^2} |V_1|^{p} \,dy.
			\end{aligned}
		\end{equation}

		We next proceed to complete the proof in three steps.

\noindent\textbf{Step 1: Bootstrap arguments and the $L^p(0,T; L^p)$ Bound.}
	Combining  \eqref{eq:est1}, \eqref{eq:force} and \eqref{eq:potentialterm}, we have 
		\begin{equation}\label{eq:diffineq}
			\begin{aligned}
&\frac{1}{p}\frac{d}{ds}  \|V_1\|_{L^p}^p+ \int_{\mathbb{R}^2} \left[(p-1)|\nabla V_1|^2 |V_1|^{p-2} + \left(\frac{9}{10}\left(\frac{1}{p}-\frac{1}{2}\right) -\epsilon\right)|V_1|^{p} \right]\, dy \\
 \leq & C(A,\lambda,p,R_0) + \left|\int_{\mathbb{R}^2} (\nabla \Pi) \cdot |V_1|^{p-2} V_1 \, dy\right|.
	\end{aligned}
	\end{equation}
	Integrating above from $0$ to $T=2 \log  \lambda$, 		thanks to the periodicity of $V_1$, we have 
	\begin{equation}\label{eq:Lpest1}
		\begin{aligned}
		&	\int_{0}^{T}  \int_{\mathbb{R}^2} \left[(p-1)|\nabla V_1|^2 |V_1|^{p-2} + \left(\frac{9}{10}\left(\frac{1}{p}-\frac{1}{2}\right) -\epsilon\right)|V_1|^{p} \right]\, dy \, ds \\
		\leq &
C(A,\lambda,p,R_0)+ \left| \int_{0}^{T}  \int_{\mathbb{R}^2} (\nabla \Pi) \cdot |V_1|^{p-2} V_1 \, dy \, ds\right|.
		\end{aligned}
	\end{equation}
	
	To estimate the last term on the right-hand side of above,	we first note that taking divergence on both sides of ${\eqref{eq:eqfordiff}}_1$, one has
	\begin{equation}
		\begin{aligned}
			-\Delta \Pi &= \textrm{ div } ( U_1\cdot \nabla V_1 + V_1\cdot \nabla  U_1 + V_1\cdot \nabla V_1 +  U_1\cdot\nabla U_1) -\nabla \cdot  F(U_0,w)\\
			&= \textrm{ div } \nabla \cdot (( U_1+V_1)   \otimes  ( U_1+V_1) ),
		\end{aligned}
	\end{equation}
	where we have used that due to \eqref{eq:eqformodifv} and $\nabla \cdot U_1=0$
	\begin{equation*}
		\nabla \cdot  F(U_0,w) = \nabla \cdot\left(\frac{\partial U_1}{\partial s}-\Delta  U_1 -\frac{1}{2} U_1 -\frac{1}{2}y\cdot \nabla  U_1\right)=0.
	\end{equation*}
	It follows from the  Calder\'{o}n–Zygmund estimates (\cite[Chapter V]{Stein93}) that there exists a constant $c(p)$ such that
	\begin{equation}\label{eq:potenest}
		\begin{aligned}
			\|\nabla \Pi\|_{L^{p}}&\leq c(p)	  \left(\| U_1\nabla  U_1\|_{L^{p}}+\|V_1\nabla  U_1\|_{L^{p}} +\| U_1 \nabla V_1\|_{L^{p}}+\|V_1\nabla V_1\|_{L^{p}} \right).
		\end{aligned}
	\end{equation}
	We  can now estimate the last term on the right-hand side of \eqref{eq:Lpest1} using \eqref{eq:Dirichletestformodifiedprofile} and \eqref{eq:potenest} together with a bootstrap argument as follows.
	Suppose that
	\begin{equation*}
		M:=\sup_{0\leq s\leq T} \|V_1(s)\|_{L^p(\mathbb{R}^2)}.
	\end{equation*}
	Then  using Lemma \ref{lem:Estimates for modified linear part}, \eqref{eq:Dirichletestformodifiedprofile} and \eqref{eq:forceest}, one has for $R_0=R_0(A,\lambda,\epsilon)$ that
	\begin{equation}\label{eq:pressterm}
		\begin{aligned}
			& \quad \left|\int_{0}^{T}\int_{\mathbb{R}^2} (\nabla \Pi) \cdot |V_1|^{p-2} V_1 \, dy \,ds\right| \leq \int_{0}^{T} \|V_1\|_{L^{p}}^{p-1} \|\nabla \Pi\|_{L^{p}}\,ds  \\
			& \leq c(p) \int_{0}^{T} \|V_1\|_{L^{p}}^{p-1} \left( \| U_1 \nabla  U_1\|_{L^{p}} + \|V_1 \nabla  U_1\|_{L^{p}} + \| U_1 \nabla V_1\|_{L^{p}} + \|V_1 \nabla V_1\|_{L^{p}}  \right)\, ds \\
			& \leq c(p)\int_{0}^{T} \|V_1\|_{L^{p}}^{p-1} \left( C(A,\lambda, p,R_0) + \epsilon \|V_1\|_{L^{p}} + \| U_1\|_{L^{\frac{2p}{2-p}}} \|\nabla V_1\|_{L^2} + \|\nabla V_1\|_{L^2} \|V_1\|_{L^{\frac{2p}{2-p}}} \right) \, ds\\
			& \leq  c(p)\int_{0}^{T} \|V_1\|_{L^{p}}^{p-1} \left( C(A,\lambda, p, R_0) + \epsilon \|V_1\|_{L^{p}} + C(A,\lambda, p, R_0) \|\nabla V_1\|_{L^2}+ C_p \|\nabla V_1\|_{L^2}^{1+\frac{p}{2}} \|V_1\|_{L^{p}}^{1-\frac{p}{2}} \right)\, ds  \\
			& \leq  c(p)\int_{0}^{T}\|V_1\|_{L^{p}}^{p-1} \left( C(A,\lambda, p, R_0) + \epsilon \|V_1\|_{L^{p}}   \right) \, ds \\
			&\quad  +c(p) \int_{0}^{T}    \|V_1\|_{L^{p}}^{2p-2} + C(A,\lambda,p, R_0)\|\nabla V_1\|_{L^2}^2 +  C_p \| V_1\|_{L^{p}}^{\frac{p}{2}}  \|\nabla V_1\|_{L^{2}}^{1+\frac{p}{2}} \, ds \\
			& \leq C(A,\lambda, p, R_0) +\left( \frac{1}{10} \left(\frac{1}{p}-\frac{1}{2}\right) + c(p) \epsilon\right) \int_{0}^{T}\|V_1\|_{L^{p}}^{p}ds + C_p\int_{0}^{T}  M^{\frac{p}{2}}  \|\nabla V_1\|_{L^{2}}^{1+\frac{p}{2}}\, ds\\
			&\leq  C(A,\lambda, p, R_0) + \left( \frac{1}{10} \left(\frac{1}{p}-\frac{1}{2}\right) +  c(p) \epsilon\right)  \int_{0}^{T}\|V_1\|_{L^{p}}^{p}ds + T^{\frac{2-p}{4}} C_p M^{\frac{p}{2}}\left(\int_{0}^{T}\|\nabla V_1\|_{L^{2}}^2 \, ds\right)^{\frac{2+p}{4}}\\
			& 		 \leq  C(A,\lambda, p, R_0) + \left( \frac{1}{10} \left(\frac{1}{p}-\frac{1}{2}\right) +  c(p) \epsilon\right) 
			\int_{0}^{T}\|V_1\|_{L^{p}}^{p}ds +  C(A,\lambda,p, R_0)M^{\frac{p}{2}},
		\end{aligned}
	\end{equation}
	where we have used Gagliardo–Nirenberg interpolation inequality 
	\begin{equation*}
		\|V_1\|_{L^{\frac{2p}{2-p}}(\mathbb{R}^2)} \leq C_p	\|\nabla V_1\|_{L^{2}(\mathbb{R}^2)}^{\frac{p}{2}} 	\|V_1\|_{L^{p}(\mathbb{R}^2)}^{1-\frac{p}{2}}, ~ 1<p<2.
	\end{equation*}
		Now we fix some $p\in(1,2)$, and let 
		\begin{equation*}
			\epsilon = \min\left\{\frac{1}{10 c(p)},\frac{1}{10}\right\}
			\cdot\left(\frac{1}{p}-\frac{1}{2}\right).
		\end{equation*}
		We then choose
	\begin{equation*}
		R_0=R_0(A,\lambda, \epsilon)  =R_0(A, \lambda, p) 
	\end{equation*}
	in \eqref{eq:defoftildev} large enough
	such that \eqref{eq:smallness} is satisfied for this $\epsilon$.
	%in Lemma \ref{lem:Estimates for modified linear part} be small enough and
	It follows that 
	\begin{equation*}
		\| U_1\|_{L^\infty}, ~	\|\nabla U_1\|_{L^\infty}\leq \epsilon \leq \min\left\{\frac{1}{10 c(p)},\frac{1}{10}\right\}
		\cdot\left(\frac{1}{p}-\frac{1}{2}\right).
	\end{equation*}
	One can now combine \eqref{eq:Lpest1} and \eqref{eq:pressterm} to have that
	\begin{equation}\label{eq:bootstrap1}
	\int_{0}^{T}  \int_{\mathbb{R}^2} |\nabla V_1|^2 |V_1|^{p-2} + |V_1|^{p} \, dy \, ds 	\leq  C(A,\lambda,p)(1+M^{\frac{p}{2}}).
	\end{equation}

		\textbf{Step 2: $L^\infty(0,T; L^p)$ Bound. }
The mean value theorem for integrals and \eqref{eq:bootstrap1} guarantees that there exists at least one specific time $s_0 \in [0, T]$ such that
\begin{equation}\label{eq:meanvalue}
\|V_1(s_0)\|_{L^p}^p \leq \frac{C(A,\lambda, p)}{T}(1+M^{\frac{p}{2}}).
\end{equation}
Integrating \eqref{eq:diffineq} from $s_0$ to $s\in[s_0, s_0+T]$ yields that
\begin{equation*}
	\begin{aligned}
		&\frac{1}{p} \left(\|V_1(s)\|_{L^p}^p - \|V_1(s_0)\|_{L^p}^p\right) + \int_{s_0}^{s} \int_{\mathbb{R}^2} (p-1)|\nabla V_1|^2 |V_1|^{p-2} + \left(\frac{1}{p} - \frac{1}{2}\right) |V_1|^p \, dy \, ds  \\
		\leq & C(A,\lambda, p)(1+M^{\frac{p}{2}}).
	\end{aligned}
\end{equation*}
Using  \eqref{eq:bootstrap1}, \eqref{eq:meanvalue} and the periodicity of $V$, one has for any $s\in[s_0,s_0+T]$ that
\begin{equation*}
	\begin{aligned}
		\|V_1(s)\|_{L^p}^p &\leq \|V_1(s_0)\|_{L^p}^p +  C(A,\lambda, p)(1+M^{\frac{p}{2}})\\
		&=  C(A,\lambda, p)(1+M^{\frac{p}{2}}).
	\end{aligned}
\end{equation*}
Since $s\in[s_0,s_0+T]$ is arbitrary, the above together with the the periodicity of $V$  shows that
\begin{equation*}
M^p=	\sup_{0\leq s\leq T}\|V_1(s)\|_{L^p(\mathbb{R}^2)}^p \leq  C(A,\lambda, p)(1+M^{\frac{p}{2}}).
\end{equation*}
Without loss of generality, we assume $M\geq 1$, otherwise, we already have a bound on $M$.
But then
\begin{equation*}
M^p\leq  2C(A,\lambda, p)M^{\frac{p}{2}},
\end{equation*}
which  implies that 
\begin{equation*}
	M\leq  C(A,\lambda, p).
\end{equation*}

\textbf{Step 3: Finish of the proof.}	
	In this step, 
we prove the $L^\infty(0,T;L^2(\mathbb{R}^2))$ estimate of $V_1$. Note that we only have 
		\begin{equation*}
			\int_0^T \|\nabla V_1(s)\|_{L^2(\mathbb R^2)}^2\,ds \leq C(A,\lambda,p),
		\end{equation*}
		 therefore one cannot  use the Gagliardo-Nirenberg inequality for each time $s\in[0,T]$ to obtain a
	$L^\infty_sL^2_y$ bound. We first derive an averaged $L^2$ bound.
		From the previous step, we know
			\begin{equation*}
		M:=\sup_{0\leq s\leq T}\|V_1(s)\|_{L^p(\mathbb R^2)}
		\leq C(A,\lambda,p),
		\quad 1<p<2.
	\end{equation*}
		By the two-dimensional Gagliardo-Nirenberg inequality, for each $s\in[0,T]$, one has
			\begin{equation*}
		\|V_1(s)\|_{L^2}
		\leq
		C_p
		\|\nabla V_1(s)\|_{L^2}^{1-\frac p2}
		\|V_1(s)\|_{L^p}^{\frac p2}.
\end{equation*}
It follows that
	\begin{equation*}
		\begin{aligned}
			\int_0^T \|V_1(s)\|_{L^2}^2\,ds
			&\leq
			C_p M^p
			\int_0^T
			\|\nabla V_1(s)\|_{L^2}^{2-p}\,ds                                      \\
			&\leq
			C_p M^p
			T^{\frac p2}
			\left(
			\int_0^T \|\nabla V_1(s)\|_{L^2}^2\,ds
			\right)^{\frac{2-p}{2}}                                                  \\
			&\leq C(A,\lambda,p).
		\end{aligned}
\end{equation*}
		Hence, there exists $s_0\in[0,T]$ such that
\begin{equation*}
		\|V_1(s_0)\|_{L^2}^2\leq C(A,\lambda,p).
\end{equation*}
		
		We next test equation \eqref{eq:eqfordiff} by $V_1$.
		We obtain
		\[
		\frac12\frac{d}{ds}\|V_1(s)\|_{L^2}^2
		+
		\|\nabla V_1(s)\|_{L^2}^2
		=
		-\int_{\mathbb R^2}(V_1\cdot\nabla U_1)\cdot V_1\,dy
		-\int_{\mathbb R^2}G\cdot V_1\,dy,
		\]
		where
		\[
		G:=U_1\cdot\nabla U_1+F(U_0,w).
		\]
		Using Lemma \ref{lem:Estimates for modified linear part} and the definition of $F(U_0,w)$, we have
		\[
		\sup_{0\leq s\leq T}\|\nabla U_1(s)\|_{L^\infty(\mathbb R^2)}
		+
		\sup_{0\leq s\leq T}\|G(s)\|_{L^2(\mathbb R^2)}
		\leq C(A,\lambda,p).
		\]
		Therefore,
		\[
		\begin{aligned}
			\frac12\frac{d}{ds}\|V_1(s)\|_{L^2}^2
			+
			\|\nabla V_1(s)\|_{L^2}^2
			&\leq
			\|\nabla U_1(s)\|_{L^\infty}
			\|V_1(s)\|_{L^2}^2
			+
			\|G(s)\|_{L^2}\|V_1(s)\|_{L^2}                                      \\
			&\leq
			C(A,\lambda,p)\|V_1(s)\|_{L^2}^2
			+
			C(A,\lambda,p).
		\end{aligned}
		\]
		By Gronwall's inequality, for any $s\in[s_0,s_0+T]$, on has
	\begin{equation*}
		\|V_1(s)\|_{L^2}^2
		\leq
		C(A,\lambda,p).
\end{equation*}
		Since $V_1$ is $T$-periodic in $s$, this implies
	\begin{equation*}
		\sup_{0\leq s\leq T}\|V_1(s)\|_{L^2}^2
		\leq C(A,\lambda,p).
	\end{equation*}
		Hence,
	\begin{equation*}
		\sup_{0\leq s\leq T}\|V_1(s)\|_{L^2(\mathbb R^2)}^2
		+
		\int_0^T\|\nabla V_1(s)\|_{L^2(\mathbb R^2)}^2\,ds
		\leq C(A,\lambda,p).
	\end{equation*}
		This completes the proof of the theorem.
	\end{proof}

	\begin{proof}[Proof of Theorem \ref{thm:energyestkey}]
   It follow from the definitions of $V$ and $V_1$ in \eqref{eq:deforre} and \eqref{eq:modre} that  
		the difference between $V$ and  $V_1$ is 
		\begin{equation*}
			V - V_1 =U_1 -U_0 = (\eta_{R_0}-1)U_0 +w.
		\end{equation*}
		It then follows from Lemmas \ref{lem:Linftyestforheat} and  \ref{lem:Estimates for modified linear part} and 
		Proposition \ref{thm:energyest} that for $p\in (1,2]$, 
		\begin{equation*}
			\sup_{0\leq s\leq T}\|V(\cdot,s) \|_{L^{p}(\mathbb{R}^2)}\leq C(A,\lambda, p), ~ \textrm{and  }\|\nabla V\|_{L^2(0,T;L^{2}(\mathbb{R}^2))}  \leq C(A,\lambda).
		\end{equation*}		
To prove the rest of Theorem \ref{thm:energyestkey}, it is sufficient to prove
\begin{equation*}
			\sup_{0\leq s\leq T} \|V(\cdot,s)\|_{H^{1}(\mathbb{R}^2)}  \leq C(A,\lambda),
		\end{equation*}
        which we formulate it as Proposition \ref{pro:uniform-H1} below.
		This finishes the proof of Theorem \ref{thm:energyestkey}.   
	\end{proof}

     \begin{proposition}[Uniform  $H^1$ estimate]
	\label{pro:uniform-H1}
	Under the same assumption of Theorem \ref{thm:energyestkey},
	we have
	\begin{equation}
		\label{eq:uniform-H1-fixed-points}
		\sup_{0\leq s\leq T}
		\|V(s)\|_{H^1(\mathbb R^2)}
		\leq C(A,\lambda).
	\end{equation}
  \end{proposition}

\begin{proof}
	Let
	\[
	\omega=\partial_1V_2-\partial_2V_1,
	\quad
	\omega_0=\partial_1U_{0,2}-\partial_2U_{0,1}.
	\]
	Taking the curl of \eqref{eq:homotopy-problem-LS}, we obtain
	\begin{equation}
		\label{eq:vorticity}
		\begin{aligned}
			\partial_s\omega-\Delta\omega-\omega
			-\frac12y\cdot\nabla\omega
			+
			(U_0+V)\cdot\nabla\omega
			=
			-( U_0+V)\cdot\nabla\omega_0.
		\end{aligned}
	\end{equation}
	Multiplying \eqref{eq:vorticity} by $\omega$ and integrating
	over $\mathbb R^2$, we have
	\begin{equation*}
		\begin{aligned}
			\frac12\frac{d}{ds}\|\omega(s)\|_{L^2}^2
			+
			\|\nabla\omega(s)\|_{L^2}^2
			-
			\frac12\|\omega(s)\|_{L^2}^2
			=
			-
			\int_{\mathbb R^2}
			( U_0+V)\cdot\nabla\omega_0\,\omega\,dy.
		\end{aligned}
	\end{equation*}
	It follows from Lemma \ref{lem:Linftyestforheat} that
	\begin{equation*}
		\sup_{0\leq s\leq T}
		\left(
		\|\nabla\omega_0(s)\|_{L^\infty}
		+
		\|U_0(s)\cdot\nabla\omega_0(s)\|_{L^2}
		\right)
		\leq C(A,\lambda).
	\end{equation*}
	Hence, using Theorem \ref{thm:energyestkey}, one obtains
	\begin{align*}
		&
		\left|
		\int_{\mathbb R^2}
		( U_0+V)\cdot\nabla\omega_0\,\omega\,dy
		\right|\\
		\leq&
		\|U_0\cdot\nabla\omega_0\|_{L^2}\|\omega\|_{L^2}
		+
		\|\nabla\omega_0\|_{L^\infty}
		\|V\|_{L^2}\|\omega\|_{L^2}\\
		\leq&
		C(A,\lambda)
		\left(
		1+\|\omega\|_{L^2}
		\right).
	\end{align*}
	Consequently, we have
	\begin{equation}
		\label{eq:vorticity-Gronwall}
		\frac{d}{ds}\|\omega(s)\|_{L^2}^2
		\leq
		C(A,\lambda)
		\left(
		1+\|\omega(s)\|_{L^2}^2
		\right).
	\end{equation}
	On the other hand,  we have
	\begin{equation*}
		\int_0^T\|\omega(s)\|_{L^2}^2\,ds
		\leq
		C\int_0^T\|\nabla V(s)\|_{L^2}^2\,ds
		\leq C(A,\lambda).
	\end{equation*}
	Therefore there exists $s_0\in[0,T]$ such that
	\[
	\|\omega(s_0)\|_{L^2}^2\leq C(A,\lambda).
	\]
	Applying Gronwall's inequality to
	\eqref{eq:vorticity-Gronwall} on $[s_0,s_0+T]$,
	and using the periodicity of $V$, one has
	\[
	\sup_{0\leq s\leq T}
	\|\omega(s)\|_{L^2}
	\leq C(A,\lambda).
	\]
	Since $V$ is divergence-free, we have
	\begin{equation*}
		\|\nabla V(s)\|_{L^2}
		=
		\|\omega(s)\|_{L^2}.
	\end{equation*}
Combining this with the uniform $L^2$ estimate of $V$ proves
	\eqref{eq:uniform-H1-fixed-points}.
\end{proof}

	\section{Construction of  DSS solutions with locally $C^2$ DSS data}
	\label{sec:passage-smooth-to-L2}
	
	In this section, we construct  DSS solutions when   the initial data is discretely
	self-similar  and locally $C^2$ by using the Leray-Schauder degree argument. The advantage here is that when the initial data is 
	locally $C^2$,  one can have faster decay for the linear part, i.e., the solution to the heat equation. This helps to establish the weighted estimates for the remainder, in which the faster decay of the linear part is needed.  This weighted estimates is a key ingredient in the proof  of the necessary compactness of the Leray-Schauder map. 
	
	Let \(u_0\in C^2_{\rm loc}(\mathbb R^2\setminus\{0\})\) be divergence-free,
	\(\lambda\)-DSS, with zero-flux, denote that
	\begin{equation*}
		B:=\|u_0\|_{C^2(\mathcal{A})},
		\textrm{ where }
		\mathcal{A}=\{1\leq |x|<\lambda\}.
	\end{equation*}
	Recall \(T=2 \log \lambda\).
	For locally $C^2$ DSS data, using the estimates for heat kernel, one can show that
	\begin{equation}
		\label{eq:U0-decay-for-LS}
		\sup_{0\leq s\leq T}
		\left|\nabla^kU_0(y,s)\right|
		\leq
		C(B,k)(1+|y|)^{-(1+k)},
		\quad k=0,1,2.
	\end{equation}
	This estimate is standard, one may refer to \cite[Lemma 2.1]{gui2026forward} for a proof.
	Recall that
	\begin{equation}\label{eq:differencenew}
		\left\{
		\begin{aligned}
			&\frac{\partial V}{\partial s}-\Delta V -\frac{1}{2}V -\frac{1}{2}y\cdot \nabla V+  (U_0+V)\cdot \nabla V + V\cdot \nabla  U_0 +\nabla \Pi=- U_0\cdot\nabla  U_0, \\
			&  \nabla \cdot V=0,
		\end{aligned} 
		\right. \textrm{ in } \mathbb{R}^2.
	\end{equation}

	\subsection{Weighted energy estimates for DSS solutions with locally $C^2$ initial data}
	\label{sec:weightedest}
\begin{proposition}[Weighted $H^1$ Estimate] \label{prop:weighted_H1est}
	Under the same assumption of Theorem \ref{thm:energyestkey} and that $B=\|u_0\|_{C^2(\mathcal{A})}<\infty$, 
	we have
	\begin{equation*} 
	\sup_{0\le s\le T}\||y|V(s)\|_{L^2}^2 + 	\| \sqrt{1+|y|^2} V \|_{L^2(0,T;H^1(\mathbb{R}^2))} \le C(B,\lambda).
	\end{equation*}
\end{proposition}
	
	\begin{proof}
		Throughout this proof, we use $C$ to denote a constant depending on $B=\|u_0\|_{C^2(\mathcal{A})}$  that may be different from line to line.
	 For simplicity,
		we test \eqref{eq:differencenew} by $|y|^2V$.	 Rigorously, one should use  $\phi = \frac{|y|^2}{1+\epsilon|y|^2} V$ as a test function and
			pass to the limit $\epsilon \to 0$. We leave this to interested readers.
	 Direct calculations using  Theorem \ref{thm:energyestkey}  give
	\begin{align*}
		-\int_{\mathbb{R}^2} \Delta V \cdot (|y|^2 V) \, dy 
		&= \int_{\mathbb{R}^2} |y|^2 |\nabla V|^2 \, dy + \int_{\mathbb{R}^2} y \cdot \nabla|V|^2 \, dy \\
		&= \| |y| \nabla V \|_{L^2}^2 - 2 \|V\|_{L^2}^2
		\geq \| |y| \nabla V \|_{L^2}^2-C.
	\end{align*}
	Moreover,
	\begin{equation*}\label{eq:drift_coercivity}
		\begin{aligned}
			& \quad \int_{\mathbb{R}^2} \left( -\frac{1}{2} V - \frac{1}{2} y \cdot \nabla V \right) \cdot (|y|^2 V) \, dy\\
			&= -\frac{1}{2} \int_{\mathbb{R}^2}  |y|^2 |V|^2 \, dy + \frac{1}{4} \int_{\mathbb{R}^2} |V|^2 \nabla \cdot (y |y|^2) \, dy\\
			&= -\frac{1}{2} \int_{\mathbb{R}^2}  |y|^2 |V|^2 \, dy +  \int_{\mathbb{R}^2} |V|^2  |y|^2 \, dy=\frac{1}{2} \int_{\mathbb{R}^2}  |y|^2 |V|^2 \, dy. 
		\end{aligned}	
	\end{equation*}
	Using $u_0$ is DSS and  $B=\|u_0\|_{C^2(\mathcal{A})}<\infty$,
	 one indeed has that $| U_0(y)| \le C(B)(1+|y|)^{-1}$.
	Hence, together with Theorem \ref{thm:energyestkey}, it follows that
	\begin{equation*}
		\begin{aligned}
			&\quad\int_{\mathbb{R}^2} (U_0\cdot\nabla V + V \cdot \nabla U_0) \cdot (|y|^2 V) \, dy\\
			&= -\int_{\mathbb{R}^2} U_0\cdot y|V|^2+V\cdot\nabla (|y|^2V)\cdot U_0 \, dy\\
			&\leq \int_{\mathbb{R}^2} |U_0\cdot y|  \cdot|V|^2 + 10|V| \cdot|y|\cdot|V| \cdot|U_0| + 10|V|\cdot|y|^2 \cdot |\nabla V|\cdot |U_0|\\
			& \leq \int_{\mathbb{R}^2} C  |V|^2 + 10|V| \cdot |y| \cdot |\nabla V|\\
			&\leq C\int_{\mathbb{R}^2}   (|V|^2 +  |\nabla V|^2) \, dy  + \frac{1}{10}\int_{\mathbb{R}^2}  |y|^2| V|^2 \, dy\\
			&\leq C + C\int_{\mathbb{R}^2}  |\nabla V|^2 \, dy   + \frac{1}{10}\int_{\mathbb{R}^2}  |y|^2| V|^2 \, dy.
		\end{aligned}
	\end{equation*}

For the nonlinear term $V\cdot \nabla V$, 
one has
	\begin{equation*}
		\begin{aligned}
			\int_{\mathbb{R}^2} (V\cdot \nabla V)\cdot |y|^2 V& = 	\frac{1}{2}\int_{\mathbb{R}^2}|y|^2   V\cdot \nabla |V|^2 = -	\int_{\mathbb{R}^2}(y\cdot V) |V|^2\\
			&\leq  C  \int_{\mathbb{R}^2}  | V|^4 \, dy + \frac{1}{10}\int_{\mathbb{R}^2}  |y|^2| V|^2 \, dy\\
			&\leq C  \|V\|_{L^2(\mathbb{R}^2)}^2\|\nabla V\|_{L^2(\mathbb{R}^2)}^2 + \frac{1}{10}\int_{\mathbb{R}^2}  |y|^2| V|^2 \, dy\\
			&\leq  C\|\nabla V\|_{L^2(\mathbb{R}^2)}^2 + \frac{1}{10}\int_{\mathbb{R}^2}  |y|^2| V|^2 \, dy.
		\end{aligned}
	\end{equation*}		
For the pressure term, the Calder\'{o}n–Zygmund  estimates imply
	\begin{equation*}
		\| q\|_{L^{2}}\leq C	  \left(\|U_0 \|_{L^{4}} + \|V\|_{L^{4}}
		\right)^2 \leq C	  \left(\|U_0 \|_{L^{4}} + \|V\|_{L^2}^{\frac{1}{2}} \|\nabla V\|_{L^2}^{\frac{1}{2}}
		\right)^2  \leq C + C\|\nabla V\|_{L^2(\mathbb{R}^2)}.
	\end{equation*}
	Hence, using H\"older inequality yields
	\begin{equation*}
		\begin{aligned}
			\int_{\mathbb{R}^2} \nabla q \cdot |y|^2 V& = 	-\int_{\mathbb{R}^2} 2q V\cdot y\leq \frac{1}{10}\int_{\mathbb{R}^2} |y|^2|V|^2 + C +  C\|\nabla V\|_{L^2(\mathbb{R}^2)}^2.
		\end{aligned}
	\end{equation*}
	Using  \eqref{eq:U0-decay-for-LS},
	one has
	\begin{equation*}
		\begin{aligned}
			 \int_{\mathbb{R}^2} (-U_0\cdot \nabla U_0 ) \cdot |y|^2 V
			& \leq C \int_{\mathbb{R}^2} |\nabla U_0|^2 
			+ \frac{1}{10} \int_{\mathbb{R}^2}  |y|^2|V|^2 \\
			& \leq C+  \frac{1}{10} \int_{\mathbb{R}^2}  |y|^2|V|^2.
		\end{aligned}
	\end{equation*}
	Combining all the above gives
	\begin{equation}\label{eq:weightedest1}
		\frac{1}{2}\frac{d}{ds}\int_{\mathbb{R}^2}|V|^2|y|^2 dy + 	\frac{1}{10}\int_{\mathbb{R}^2}|V|^2|y|^2 dy + \int_{\mathbb{R}^2}|\nabla V|^2|y|^2 dy\leq C(B,\lambda) + C\int_{\mathbb{R}^2}  |\nabla V|^2 \, dy.
	\end{equation}
	
		%\textbf{Step 1: Weighted $L^2(0,T; H^1)$ Bound of $V$. }
		Integrating \eqref{eq:weightedest1} from $[0,T]$, using the periodicity of $V$ and Theorem \ref{thm:energyestkey}, one has
		\begin{equation}\label{eq:integralest} 
			\begin{aligned}
	\frac{1}{10}\int_{0}^{T}\int_{\mathbb{R}^2}|V|^2|y|^2 dy + \int_{0}^{T} \int_{\mathbb{R}^2}|\nabla V|^2|y|^2 dy\leq C(B,\lambda).
			\end{aligned}
		\end{equation}
		
	%	\textbf{Step 2: Weighted  $L^\infty(0,T; L^2)$ Bound of $V$. }
	The mean value theorem for integrals and \eqref{eq:integralest} guarantees that there exists some $s_0 \in [0, T]$ such that
	\begin{equation*}
		\||y|V(s_0)\|_{L^2}^2 \leq C(B,\lambda).
	\end{equation*}
	Integrating \eqref{eq:weightedest1} from $s_0$ to $s\in[s_0,s_0+T]$ yields that
	\begin{equation*}
		\begin{aligned}
			&\frac{1}{2} \left(\||y|V(s)\|_{L^2}^2 - \||y|V(s_0)\|_{L^2}^2\right) \leq C(B,\lambda).
		\end{aligned}
	\end{equation*}
	This together with \eqref{eq:integralest} and the periodicity of $V$ implies that
	\begin{equation*}
		\sup_{0\le s\le T}\||y|V(s)\|_{L^2}^2 \leq C(B,\lambda).
	\end{equation*}
This finishes the proof of the proposition.
		\end{proof}

		 The above weighted energy estimates are indeed sufficient for us to  construct discretely self-similar solutions. The following theorem is the main result in this section.
		\begin{theorem}[Existence of discretely self-similar solutions with $C^2_{loc}$ initial data]
		Let $u_0(x)\in C^{2}_{loc}(\R^2\setminus\{0\})$ be a divergence-free, $\lambda$-DSS vector field  in $\R^2 \setminus \{0\}$, such that
		\begin{equation*}
			\int_{ \mathbb{S}^1} u_0\cdot n d\sigma = 0 .
		\end{equation*}
		Denote 
		$B=\|u_0\|_{C^{2}(\mathcal{A})}<+\infty$, where $\mathcal{A} =\{x \in \R^2 : 1 \leq |x| < \lambda\}$.
		Then	there exists at least one discretely  self-similar solution $u (x,t)$ to the Cauchy problem \eqref{eq:NS}-\eqref{eq:initial}.
		\end{theorem}

\subsection{Construction of DSS solutions for locally $C^2$ initial data}
\label{sec:LS-construction-smooth-data}

The main goal of this section is to construct solutions by applying the Leray-Schauder fixed point theorem.
 The key step is to find  appropriate function set-up and prove  enough a priori estimates.
Assume that
\(
u_0\in C^2_{\mathrm{loc}}(\mathbb R^2\setminus\{0\})
\)
is $\lambda$-DSS and  divergence-free with zero-flux. Recall that $U_0$ is the
profile defined by
\[
e^{t\Delta}u_0(x)
=
\frac1{\sqrt t}
U_0\left(\frac{x}{\sqrt t}, \log  t\right).
\]

We consider the following problem
\begin{equation}
	\label{eq:homotopy-problem-LS}
	\left\{
	\begin{aligned}
		\partial_sV-\Delta V-\frac12V-\frac12y\cdot\nabla V
		&
		+\nabla\Pi=-	(\sigma U_0+V)\cdot\nabla(\sigma U_0+V),\\
		\nabla\cdot V&=0,\\
		V(\cdot,s+T)&=V(\cdot,s),
	\end{aligned}
	\right.
	\quad 0\leq\sigma\leq1.
\end{equation}

{In order to apply the Leray-Schauder fixed point argument,} we first fix $\delta \in (\frac{1}{3}, \frac{1}{2})$.
Let $X$ be the space of all divergence-free, $T$-periodic vector
fields $V$ such that
\[
(1+|y|^2)^{\delta/2}V
\in L^5(\mathbb R^2\times(0,T)),
\]
with norm
\begin{equation*}
	\|V\|_X
	=\left\|(1+|y|^2)^{\delta/2}V\right\|_{L^5(\mathbb R^2\times(0,T)}
	=\left(
	\int_0^T\int_{\mathbb R^2}
	(1+|y|^2)^{5\delta/2}|V(y,s)|^5\,dy\,ds
	\right)^{1/5}.
\end{equation*}
The space $X$ is a Banach space. It follows from decay estimates
\eqref{eq:U0-decay-for-LS} and $\delta<1/2$ that $U_0\in X$.

We now
let $F=F(y,s)$ be a tensor field, periodic in time $s$ with period $T$. 
Let $Z_F$ solves
\begin{equation}
	\label{eq:forced-linear-profile}
	\left\{
	\begin{aligned}
		&\partial_sZ_F-\Delta Z_F-\frac12Z_F
		-\frac12y\cdot\nabla Z_F+\nabla P_F
		=-\nabla\cdot F,\\
		&\nabla\cdot Z_F=0.
	\end{aligned}
	\right.
\end{equation}
Then  $Z_F$ can be solved by using the heat kernel as
\begin{equation}
	\label{eq:ZF-definition}
	\begin{aligned}
		Z_F(y,s)
		:=
		-\int_0^1
		e^{(1-\tau)\Delta}\mathbb P\nabla\cdot
		\left[
		\tau^{-1}F\left(\frac{\cdot}{\sqrt \tau},s+ \log  \tau \right)
		\right](y)\,d\tau,
	\end{aligned}
\end{equation}
where  $\mathbb P$ is
the Leray projection. It is obvious that $Z_F$ is periodic in $s$ with period $T$. 
 Equivalently, if one writes
\[
\widetilde F(x,t)
=
\frac1tF\left(\frac{x}{\sqrt t}, \log  t\right),
\]
then
\begin{equation*}
	\begin{aligned}
		\frac1{\sqrt t}
		Z_F\left(\frac{x}{\sqrt t}, \log  t\right)
		=
		-\int_0^t
		e^{(t-\tau)\Delta}\mathbb P\nabla\cdot
		\widetilde F(\tau)\,d\tau.
	\end{aligned}
\end{equation*}

Finally,
for any  $V\in X$ and   $\sigma\in[0,1]$, we define the tensor field
\begin{equation*}
	F_{V,\sigma}= (\sigma U_0+V)  \otimes  (\sigma U_0+V),
\end{equation*}
 and we let
\begin{equation*}
	\mathcal K(V,\sigma)=Z_{F_{V,\sigma}}.
\end{equation*}
It is easy to see that
the fixed point of $\mathcal K(V,\sigma)=V$ solves \eqref{eq:homotopy-problem-LS}. The goal is to solve the fixed point problem
	$\mathcal K(V,1)=V$.
We first show that the map $\mathcal K$ is continuous and compact. We have

\begin{proposition}\label{pro:compactofmap}
	The map 
	\begin{equation*}
		\mathcal{K}:X\times[0,1] \to X
	\end{equation*}
	is continuous and compact.
	\end{proposition}
To prove this proposition, 
we need some weighted estimates for the inhomogeneous Stokes equations \eqref{eq:forced-linear-profile}. 
\begin{lemma}[Weighted  Stokes estimate]
	\label{lem:weighted-profile-Stokes}
	Let $0\leq\beta<1$. Let $Z_F$ be defined in \eqref{eq:ZF-definition}. There exists a constant $C=C(\beta,T)$ such that
	\begin{equation}
		\label{eq:weighted-profile-Stokes-estimate}
		\begin{aligned}
			\left\|
			(1+|y|^2)^{\beta/2}Z_F
			\right\|_{L^6(\mathbb R^2\times(0,T))}\leq
			C
			\left\|
			(1+|y|^2)^{\beta/2}F
			\right\|_{L^{\frac{5}{2}}(\mathbb R^2\times(0,T))}.
		\end{aligned}
	\end{equation}
	Moreover, for every $R>0$, there exists a constant $C=C(R, \beta,T)$ such that
	\begin{equation}
		\label{eq:local-profile-Stokes-estimate}
		\begin{aligned}
			\,&\|Z_F\|_{L^{\frac{5}{2}}(0,T;W^{1,\frac{5}{2}}(B_R))}
			+
			\|\partial_sZ_F\|_{L^{\frac{5}{2}}(0,T;W^{-1,\frac{5}{2}}(B_R))}\\
			\leq\,&
			C(R,\beta,T)
			\left\|
			(1+|y|^2)^{\beta/2}F
			\right\|_{L^{\frac{5}{2}}(\mathbb R^2\times(0,T))}.
		\end{aligned}
	\end{equation}
\end{lemma}
We give the proof of this lemma in Appendix \ref{sec:App}. With this lemma in hand, we can prove Proposition \ref{pro:compactofmap}.
\begin{proof}[Proof of Proposition \ref{pro:compactofmap}]
Recall that
\begin{equation*}
F_{V,\sigma}= (\sigma U_0+V)  \otimes  (\sigma U_0+V),
\end{equation*}
	Then 
	\begin{equation*}
		\begin{aligned}
			&\left\|
			(1+|y|^2)^\delta 	F_{V,\sigma}
			\right\|_{L^{\frac{5}{2}}(\mathbb R^2\times(0,T))}\leq
			\left\|
			(1+|y|^2)^{\delta/2}(\sigma U_0+V)
			\right\|_{L^5(\mathbb R^2\times(0,T))}^2 = \|\sigma U_0+V\|^2_X.
		\end{aligned}
	\end{equation*}
	Since $2\delta<1$, one can apply  Lemma \ref{lem:weighted-profile-Stokes} with
	$\beta=2\delta$ together with the above to have
	\begin{equation}
		\label{eq:K-strong-weighted-L6}
		\begin{aligned}
			\left\|
			(1+|y|^2)^\delta\mathcal K(V, \sigma)
			\right\|_{L^6(\mathbb R^2\times(0,T))} &\leq C(\delta,T) \left\|
			(1+|y|^2)^\delta 	F_{V,\sigma}
			\right\|_{L^{\frac{5}{2}}(\mathbb R^2\times(0,T))}\\
			&\leq
			C(\delta,T)\left(\|U_0\|_X+\|V\|_X\right)^2.
		\end{aligned}
	\end{equation}
	Using Hölder's inequality and $\delta>\frac{1}{3}$, we obtain
	\begin{equation}
		\label{eq:K-maps-into-X}
		\begin{aligned}
			\|\mathcal K(V, \sigma)\|_X
			&\leq
			\left\|
			(1+|y|^2)^\delta\mathcal K(V, \sigma)
			\right\|_{L^6}
			\left\|
			(1+|y|^2)^{-\delta/2}
			\right\|_{L^{30}(\mathbb R^2\times(0,T))}\\
			&\leq C(\delta,T) \left\|
			(1+|y|^2)^\delta\mathcal K(V, \sigma)		\right\|_{L^6}\leq C(\delta,T)\left(\|U_0\|_X+\|V\|_X\right)^2.
		\end{aligned}
	\end{equation}
Hence
	$\mathcal K$ maps bounded sets of $X\times[0,1]$ into bounded sets of $X$.
	
	We next prove the continuity of $\mathcal{K}$.
	Let $V_1,V_2\in X$ and $\sigma_1,\sigma_2\in[0,1]$. 
It follows from direct calculations that
\begin{equation}\label{eq:difference1}
	\begin{aligned}
		&\left\|
		(1+|y|^2)^\delta
		(F_{V_1,\sigma_1} - 	F_{V_2,\sigma_2})
		\right\|_{L^{\frac{5}{2}}(\mathbb R^2\times(0,T))}\\
		\leq&
		\left(2\|U_0\|_X +\|V_1\|_X+\|V_2\|_X\right)
		\left(
		\|V_1-V_2\|_X
		+|\sigma_1-\sigma_2|\|U_0\|_X
		\right).
	\end{aligned}
\end{equation}
	Lemma \ref{lem:weighted-profile-Stokes} together with \eqref{eq:K-maps-into-X} and \eqref{eq:difference1}
 yields that
	\begin{equation*}
		\begin{aligned}
		\|\mathcal{K}(V_1,\sigma_1) - \mathcal{K}(V_2,\sigma_2)\|_X &\leq C(\delta,T)\left\|
		(1+|y|^2)^\delta(\mathcal K(V_1, \sigma_1)	 -	\mathcal K(V_2, \sigma_2)	 )\right\|_{L^6}\\
		&\leq  C(\delta,T) \left\|
		(1+|y|^2)^\delta 	(F_{V_1,\sigma_1}-F_{V_2,\sigma_2})
		\right\|_{L^{\frac{5}{2}}(\mathbb R^2\times(0,T))}\\
		&\leq C(\delta,T)\left(2\|U_0\|_X +\|V_1\|_X+\|V_2\|_X\right)
		\left(
		\|V_1-V_2\|_X
		+|\sigma_1-\sigma_2|\|U_0\|_X
		\right),
			\end{aligned}
	\end{equation*}
	 which shows the continuity of $\mathcal{K}$ on $X\times[0,1]$.
	
	Finally, we prove that $\mathcal{K}$ is a compact mapping. Let $\{V_n\}$ be bounded sequence in $X$, and
	$\sigma_n\to\sigma\in [0,1]$. 
It suffices to show that $\mathcal{K}(V_n,\sigma_n)$ has a convergent subsequence in $X$. After passing to a subsequence, we may assume that there exists some $Z\in X$ such that 
$\mathcal{K}(V_n,\sigma_n)\to Z$ weakly in $X$.
For every $R>0$, Lemma \ref{lem:weighted-profile-Stokes} shows 
	that $\{\mathcal{K}(V_n,\sigma_n)\}$ is bounded in
	$L^{\frac{5}{2}}(0,T;W^{1,\frac{5}{2}}(B_R))$.
	and $\{\partial_s \mathcal{K}(V_n,\sigma_n)\}$ is bounded in
$L^{\frac{5}{2}}(0,T;W^{-1,\frac{5}{2}}(B_R))$.
	Since
	\begin{equation*}
		W^{1,\frac{5}{2}}(B_R)
		\Subset L^{\frac{5}{2}}(B_R)
		\hookrightarrow W^{-1,\frac{5}{2}}(B_R),
	\end{equation*}
	the Aubin-Lions lemma shows that $\{\mathcal{K}(V_n,\sigma_n)\}$ is compact subset of $L^{\frac{5}{2}}(B_R\times(0,T))$.
	Hence,
	\begin{equation*}
		\mathcal{K}(V_n,\sigma_n) \to Z
		\quad\text{strongly in }
		L^{\frac{5}{2}}(B_R\times(0,T)).
	\end{equation*}
	Due to \eqref{eq:K-strong-weighted-L6}, we have
	\begin{equation*}
		\sup_n
		\left\|
		(1+|y|^2)^\delta \mathcal{K}(V_n,\sigma_n)
		\right\|_{L^6(\mathbb R^2\times(0,T))}<\infty.
	\end{equation*}
Interpolation with  $L^6(\mathbb R^2\times(0,T))$ gives
	\begin{equation*}
		\mathcal{K}(V_n,\sigma_n) \to  Z
		\textrm{ strongly in }
		L^5(B_R\times(0,T)).
	\end{equation*}
To show the convergence
\begin{equation*}
	\mathcal{K}(V_n,\sigma_n) \to Z \textrm{ in } X,
\end{equation*}
which would yield the compactness of $\mathcal{K}$, we only need to prove uniform tail estimates in $X$.
Indeed, H\"older's inequality gives
\begin{equation*}
	\begin{aligned}
		& \left\|(1+|y|^2)^{\frac{\delta}{2}}\mathcal{K}(V_n,\sigma_n)\right\|_{L^5(\{|y|>R\}\times(0,T))}\\
		\leq& \left\|(1+|y|^2)^{\delta}\mathcal{K}(V_n,\sigma_n)\right\|_{L^6(\{|y|>R\}\times(0,T))}\left\|(1+|y|^2)^{-\frac{\delta}{2}}\right\|_{L^{30}(\{|y|>R\}\times(0,T))}\leq CR^{\frac{1}{15}-\delta}
		\end{aligned}
\end{equation*}
	Since $\delta>1/3$, the right-hand side above tends to $0$ as $R\to\infty$, uniformly in $n$. This finishes the proof.
\end{proof}

Now the major task is to prove the uniform bounds for the fixed points in $X$.

\begin{lemma}[Uniform bounds for fixed-points]
	\label{lem:uniform-fixed-point-bound-LS}
	There exists $C=C(B,\lambda,\delta)$ such that every fixed point of 
	\(V=\mathcal K(V, \sigma),
	 0\leq\sigma\leq1\)
	satisfies $\|V\|_X\leq C$.
\end{lemma}

\begin{proof}
	By Theorem \ref{thm:energyestkey} and
	Proposition \ref{prop:weighted_H1est}, one has
	\begin{equation*}
	\sup_{0\leq s\leq T}
	\left\|
	(1+|y|^2)^{1/2}V(s)
	\right\|_{L^2(\mathbb{R}^2)}
	+
	\left\|
	(1+|y|^2)^{1/2}V
	\right\|_{L^2(0,T;H^1(\mathbb R^2))}
	\leq C(B,\lambda).
	\end{equation*}
	The  Ladyzhenskaya inequality  gives
	\begin{equation*}
		\begin{aligned}
			&\left\|
			(1+|y|^2)^{1/2}V
			\right\|_{L^4(\mathbb R^2\times(0,T))}^4\\
			\leq&
			C
			\sup_{0\leq s\leq T}
			\left\|
			(1+|y|^2)^{1/2}V(s)
			\right\|_{L^2}^2
			\int_0^T
			\left\|
			(1+|y|^2)^{1/2}V(s)
			\right\|_{H^1}^2\,ds\\
			\leq & C(B,\lambda).
		\end{aligned}
	\end{equation*}
It follows from  Lemma \ref{pro:uniform-H1} and the
Sobolev inequality that, for every finite $q>2$,
	\begin{equation*}
		\sup_{0\leq s\leq T}
		\|V(s)\|_{L^q(\mathbb R^2)}
		\leq C(B,\lambda,q).
	\end{equation*}
	For the fixed number $\delta\in(1/3,1/2)$, we define
	\(	q_\delta
	=
	\frac{20(1-\delta)}{4-5\delta}>2\) so that
	\[
	\frac15
	=
	\frac{\delta}{4}
	+
	\frac{1-\delta}{q_\delta}.
	\]
	Noting
	\begin{equation*}
		(1+|y|^2)^{\delta/2}|V|
		=
		\left(
		(1+|y|^2)^{1/2}|V|
		\right)^\delta
		|V|^{1-\delta},
	\end{equation*} 
using
	H\"older's inequality then gives
	\begin{align*}
		\|V\|_X
		&=
		\left\|
		(1+|y|^2)^{\delta/2}V
		\right\|_{L^5(\mathbb R^2\times(0,T))}\\
		&\leq
		\left\|
		(1+|y|^2)^{1/2}V
		\right\|_{L^4(\mathbb R^2\times(0,T))}^{\delta}
		\|V\|_{L^{q_\delta}(\mathbb R^2\times(0,T))}^{1-\delta}\\
		&\leq
		C(B,\lambda,\delta).
	\end{align*}
	This proves the desired uniform bound in $X$.
\end{proof}

Now we are ready to apply the Leray-Schauder fixed point theorem to prove the existence of solutions.

\begin{theorem}
	Under the same assumption of Theorem \ref{thm:energyestkey} and that $B=\|u_0\|_{C^2(\mathcal{A})}<\infty$, 
there exists a smooth,
	$T$-periodic solution $V\in X$ of
	\begin{equation}
		\label{eq:final-profile-equation-LS}
		\left\{
		\begin{aligned}
			\partial_sV-\Delta V-\frac12V-\frac12y\cdot\nabla V
			&+
			(U_0+V)\cdot\nabla(U_0+V)+\nabla\Pi=0,\\
			\nabla\cdot V&=0,\\
			V(\cdot,s+T)&=V(\cdot,s).
		\end{aligned}
		\right.
	\end{equation}
	Moreover, there eixsts a constant $C=C(A,\lambda)$ such that
	\begin{equation*}
	\sup_{0\leq s\leq T} \|V(\cdot,s)\|_{L^2(\mathbb{R}^2)}+	\|V\|_{L^2(0,T;H^{1}(\mathbb{R}^2))}  \leq C(A,\lambda).
	\end{equation*}
\end{theorem}

\begin{proof}
	By Proposition~\ref{pro:compactofmap}, the map
	\[
	\mathcal K:X\times[0,1]\longrightarrow X
	\]
	is continuous and compact.
All possible solutions to
	$V=\mathcal K(V, \sigma)$
	has 
	uniform bounds in $X$ by
	Lemma~\ref{lem:uniform-fixed-point-bound-LS}. 
	When $\sigma=0$, it easy to see 
		\[
	V=\mathcal K(V, 0)
	\]
	has a unique solution $V=0$ in $X$.
	The Leray-Schauder fixed point theorem implies that
	we have at least one solution to
		\[
	V=\mathcal K(V, 1),
	\]
which is just a
	solution of \eqref{eq:final-profile-equation-LS}. The remaining statement holds due to Theorem \ref{thm:energyest}.
%	the local
%	regularity argument  for two-dimensional NS equations shows that $V$ is smooth.
\end{proof}

\section{Discretely self-similar solutions with rough initial data}
\label{sec:roughdata}

In this section, we 
 construct  DSS solutions with
initial data  belonging  to
\(L^2_{\mathrm{loc}}(\mathbb R^2\setminus\{0\})\) by  approximating the initial data 
with locally $C^2$ functions.
The essential point for the approximation argument to work is that the estimates in
Theorem~\ref{thm:energyestkey} depend only on the
\(L^2(\mathcal A)\)-norm of the initial datum, and not on its higher
regularity. 

\subsection{Approximation of the $L^2_{\mathrm{loc}}$-initial data }

We first give  an  approximation result for  divergence-free vector field that preserves all the relevant
structures.

\begin{lemma}
	Let
	\(
	u_0\in L^2_{\mathrm{loc}}(\mathbb R^2\setminus\{0\})
\)
	be a divergence-free, \(\lambda\)-DSS vector field and that
	\(\int_{\mathbb{S}^1}	u_0\cdot n d\sigma =0\). Then
	there exists a sequence
	\(
	u_0^{(n)}
	\in C^\infty_{\mathrm{loc}}(\mathbb R^2\setminus\{0\})
	\)
	such that each \(u_0^{(n)}\) is divergence-free, 	\(\lambda\)-DSS, satisfying 	\(\int_{\mathbb{S}^1}	u_0\cdot n d\sigma =0\) and that
	\begin{equation}
		\label{eq:smooth-DSS-approximation}
		u_0^{(n)}\longrightarrow u_0
		\quad\text{strongly in }L^2(\mathcal A).
	\end{equation}
	Moreover, it holds that
	\begin{equation}
		\label{eq:smooth-DSS-uniform-bound}
		\sup_{n\geq1}
		\|u_0^{(n)}\|_{L^2(\mathcal A)}
		\leq
		C\|u_0\|_{L^2(\mathcal A)}.
	\end{equation}
\end{lemma}

\begin{proof}
	Write \(x=(r,\theta)\) in polar coordinates and set
	\(\tau= \log  r\). Since \(u_0\) is \(\lambda\)-DSS, it can be written as
	\[
	u_0(r,\theta)
	=
	\frac1r
	\bigl(
	a(\tau,\theta)e_r
	+
	b(\tau,\theta)e_\theta
	\bigr),
	\]
	where \(a\) and \(b\) are periodic in \(\tau\) and $\theta$ with period
	\( \log \lambda\) and $2\pi$, respectively.
	Furthermore,
	\[
	\|u_0\|_{L^2(\mathcal A)}^2
	=
	\int_0^{ \log \lambda}\int_0^{2\pi}
	\bigl(|a|^2+|b|^2\bigr)\,d\theta\,d\tau.
	\]
	The divergence-free condition is equivalent to
	\[
	\partial_\tau a+\partial_\theta b=0
	\]
	 on 
	\(
	\mathbb T_{ \log \lambda}\times\mathbb \mathbb{S}^1
\).
	The zero-flux condition is
	\[
	\int_0^{2\pi}a(\tau,\theta)\,d\theta=0.
	\]
	
	Let \(\varphi_n\) be an approximation of identity  and define
	\[
	a_n=\varphi_n*a,
	\quad
	b_n=\varphi_n*b.
	\]
	Then both $a_n$ and $b_n$ are periodic in \(\tau\) and $\theta$ with period
	\( \log \lambda\) and $2\pi$, respectively.
	And that
	\[
	\partial_\tau a_n+\partial_\theta b_n=0,
	\quad
	\int_0^{2\pi}a_n(\tau,\theta)\,d\theta=0.
	\]
	Set
\begin{equation*}
			u_0^{(n)}(r,\theta)
		=
		\frac1r
		\bigl(
		a_n( \log  r,\theta)e_r
		+
		b_n( \log  r,\theta)e_\theta
		\bigr).
\end{equation*}
	Then \(u_0^{(n)}\) is \(\lambda\)-DSS, smooth away from the origin, divergence-free on $\mathbb{R}^2\setminus\{0\}$,
	and has zero-flux. Finally, \eqref{eq:smooth-DSS-approximation} and \eqref{eq:smooth-DSS-uniform-bound} follow from the
	standard properties of  convolution.
\end{proof}

For each \(n\), let \(U_{0,n}\) denote the corresponding heat
profile:
\begin{equation*}
	U_{0,n}(y,s)
	=
	\sqrt t\,
	\bigl(e^{t\Delta}u_0^{(n)}\bigr)(x),
	\quad
	y=\frac{x}{\sqrt t},
	\quad
	s= \log  t.
\end{equation*}
Similarly, let \(U_0\) be the heat profile associated with \(u_0\).
We then have the convergence of the linear profiles.
\begin{lemma}
	For every \(q>2\), we have that
	\begin{equation}
		\label{eq:linear-profile-convergence}
		\sup_{0\leq s\leq T}
		\|U_{0,n}(\cdot,s)-U_0(\cdot,s)\|_{L^q(\mathbb R^2)}
		\longrightarrow0
	\end{equation}
	as \(n\to\infty\). In particular, this implies
	\begin{equation}
		\label{eq:linear-profile-L4-convergence}
		U_{0,n}\longrightarrow U_0
		\quad\text{strongly in }
		L^4((0,T)\times\mathbb R^2),
	\end{equation}
and that
	\begin{equation}
		\label{eq:linear-profile-uniform-L4}
		\sup_n\sup_{0\leq s\leq T}
		\|U_{0,n}(\cdot,s)\|_{L^4(\mathbb R^2)}
		\leq C(A,\lambda),
	\end{equation}
	where
	\(
	A=\|u_0\|_{L^2(\mathcal A)}.
\)
\end{lemma}

\begin{proof}
	For a \(\lambda\)-DSS function \(f\), one has
	\[
	\|f\|_{L^{2,\infty}(\mathbb R^2)}
	\leq C(\lambda)\|f\|_{L^2(\mathcal A)}.
	\]
	Consequently, \eqref{eq:smooth-DSS-approximation} implies
	\[
	\|u_0^{(n)}-u_0\|_{L^{2,\infty}(\mathbb R^2)}
	\longrightarrow0.
	\]
	For \(s\in[0,T]\), define
	\[
	\widetilde u_{0,n}(y,s)
	=
	e^{s/2}u_0^{(n)}(e^{s/2}y),
	\quad
	\widetilde u_0(y,s)
	=
	e^{s/2}u_0(e^{s/2}y),
	\]
 and
	\[
	U_{0,n}(\cdot,s)=e^\Delta\widetilde u_{0,n}(\cdot,s),
	\quad
	U_0(\cdot,s)=e^\Delta\widetilde u_0(\cdot,s).
	\]
	The heat kernel estimate from \(L^{2,\infty}\) to \(L^q\), \(q>2\),
	therefore gives for  \(s\in[0,T]\) that
	\[
	\begin{aligned}
		\|U_{0,n}(\cdot,s)-U_0(\cdot,s)\|_{L^q}
		&\leq
		C(q)
		\|\widetilde u_{0,n}(\cdot,s)
		-\widetilde u_0(\cdot,s)\|_{L^{2,\infty}}  \\
		&\leq
		C(q,\lambda)
		\|u_0^{(n)}-u_0\|_{L^2(\mathcal A)}.
	\end{aligned}
	\]
	This estimate is uniform in \(s\in[0,T]\), and hence proves
	\eqref{eq:linear-profile-convergence}. The remaining conclusions follow
	immediately.
\end{proof}

\subsection{Existence of DSS solutions for $L^2_{loc}$ initial data}

For each \(n\), by the results in Section \ref{sec:LS-construction-smooth-data}, there exists a smooth
\(T\)-periodic profile \(V_n\) solves the following system
\begin{equation}
	\label{eq:Vn-equation}
	\left\{
	\begin{aligned}
		\partial_sV_n-\Delta V_n-\frac12V_n-\frac12y\cdot\nabla V_n
		&+
		(U_{0,n}+V_n)\cdot\nabla(U_{0,n}+V_n)
		+\nabla\Pi_n=0, \\
		\nabla\cdot V_n&=0,\\
		V_n(\cdot,s+T)&=V_n(\cdot,s).
	\end{aligned}
	\right.
\end{equation}
By Theorem \ref{thm:energyestkey} and
\eqref{eq:smooth-DSS-uniform-bound}, we have
\begin{equation}
	\label{eq:Vn-uniform-energy}
	\sup_n
	\left[
	\sup_{0\leq s\leq T}\|V_n(s)\|_{L^2}^2
	+
	\int_0^T\|V_n(s)\|_{H^1}^2\,ds
	\right]
	\leq C(A,\lambda),
\end{equation}
where $A=\|u_0\|_{L^2(\mathcal A)}$.
Moreover, for every fixed \(1<p<2\), it holds that
\begin{equation}\label{eq:Vn-uniform-energy1}
	\sup_n\sup_{0\leq s\leq T}
	\|V_n(s)\|_{L^p}
	\leq C(A,\lambda,p).
\end{equation}
The Ladyzhenskaya inequality gives
\begin{equation*}
	\|V_n(s)\|_{L^4}^4
	\leq
	C\|V_n(s)\|_{L^2}^2
	\|\nabla V_n(s)\|_{L^2}^2.
\end{equation*}
Therefore,
\begin{equation}
	\label{eq:Vn-uniform-L4}
	\sup_n
	\int_0^T\int_{\mathbb R^2}|V_n|^4\,dy\,ds
	\leq C(A,\lambda).
\end{equation}

We next derive the time derivative estimates needed to get the strong convergence of $V_n$.

\begin{lemma}[Estimate for the time derivative]
	\label{lem:local-time-derivative}
	For every \(R>0\), we have
	\begin{equation*}
		\sup_n
		\|\partial_sV_n\|_{L^2(0,T;H^{-1}(B_R))}
		\leq C(A,\lambda,R).
	\end{equation*}
\end{lemma}

\begin{proof}
	The  pressure is given by
	\begin{equation*}
		\Pi_n
		=
		R_iR_j
		\left[
		(U_{0,n}+V_n)_i
		(U_{0,n}+V_n)_j,
		\right]
	\end{equation*}
	where $R_i$ is the Riesz transformation. 
	 By the Calder\'{o}n-Zygmund
	estimates,
	\[
	\|\Pi_n(s)\|_{L^2}
	\leq
	C\|U_{0,n}(s)+V_n(s)\|_{L^4}^2.
	\]
	Using \eqref{eq:linear-profile-uniform-L4} and
	\eqref{eq:Vn-uniform-L4}, we obtain
	\begin{equation}
		\label{eq:pressure-uniform-L2}
		\sup_n
		\int_0^T\|\Pi_n(s)\|_{L^2}^2\,ds
		\leq C(A,\lambda).
	\end{equation}
	Fix \(R>0\). On \(B_R\), the terms $\Delta V_n$, $V_n$, and $y\cdot\nabla V_n$
	are uniformly bounded in \(L^2(0,T;H^{-1}(B_R))\). Moreover,
	\[
(U_{0,n}+V_n)\cdot\nabla(U_{0,n}+V_n)
	=
	\nabla\cdot
	\left[
	(U_{0,n}+V_n)  \otimes  (U_{0,n}+V_n)
	\right]
	\]
	is uniformly bounded in \(L^2(0,T;H^{-1}(B_R))\), since
	\(
	U_{0,n}+V_n
	\)
	is uniformly bounded in \(L^4((0,T)\times\mathbb R^2)\). Finally,
	\(\nabla\Pi_n\) is uniformly bounded in
	\(L^2(0,T;H^{-1}(B_R))\) by
	\eqref{eq:pressure-uniform-L2}. The conclusion follows from
	\eqref{eq:Vn-equation}.
\end{proof}

\begin{proposition}[Strong convergence of the approximate sequences]
	There exists a subsequence, still denoted by \(\{V_n\}\), and a
	\(T\)-periodic vector field \(V\) such that
	\begin{equation*}
		V_n\rightharpoonup^\ast V
		\quad\text{weakly-* in }
		L^\infty(0,T;L^2(\mathbb R^2)),
	\end{equation*}
	\begin{equation*}
		V_n\rightharpoonup V
		\quad\text{weakly in }
		L^2(0,T;H^1(\mathbb R^2)),
	\end{equation*}
	and, for every \(R>0\),
	\begin{equation}
		\label{eq:Vn-local-strong-L2}
		V_n\longrightarrow V
		\quad\text{strongly in }
		L^2((0,T)\times B_R).
	\end{equation}
	Furthermore, we have
	\begin{equation}
		\label{eq:V-limit-energy}
		V\in
		L^\infty(0,T;L^2(\mathbb R^2))
		\cap
		L^2(0,T;H^1(\mathbb R^2)),
	\end{equation}
	and
	\begin{equation}
		\label{eq:V-limit-energy-bound}
		\sup_{0\leq s\leq T}\|V(s)\|_{L^2}^2
		+
		\int_0^T\|V(s)\|_{H^1}^2\,ds
		\leq C(A,\lambda).
	\end{equation}
	For every \(1<p<2\), one has
	\begin{equation}
		\label{eq:V-limit-Lp}
		\operatorname*{ess\,sup}_{0\leq s\leq T}
		\|V(s)\|_{L^p}
		\leq C(A,\lambda,p).
	\end{equation}
\end{proposition}

\begin{proof}
	The weak and weak-* convergences follow from
	\eqref{eq:Vn-uniform-energy}. For each \(R>0\), the sequence
	\(\{V_n\}\) is bounded in \(L^2(0,T;H^1(B_R))\), while
	Lemma~\ref{lem:local-time-derivative} gives a uniform bound for
	\(\{\partial_sV_n\}\) in \(L^2(0,T;H^{-1}(B_R))\). Since
	\[
	H^1(B_R)\subset\subset L^2(B_R)\hookrightarrow H^{-1}(B_R),
	\]
	the Aubin-Lions lemma yields strong convergence in
	\(L^2((0,T)\times B_R)\). A diagonal argument yields a 
	subsequence satisfying \eqref{eq:Vn-local-strong-L2}.
	The weak limit is therefore also \(T\)-periodic. The estimates
	\eqref{eq:V-limit-energy-bound} and \eqref{eq:V-limit-Lp} follow from
	 lower semicontinuity for the weak convergence and the corresponding uniform bounds for $V_n$ in \eqref{eq:Vn-uniform-energy} and \eqref{eq:Vn-uniform-energy1}.
\end{proof}

We next show that the limit $V$ solves \eqref{eq:difference}.
Let
\(
\Phi\in
C^\infty(\mathbb T_T;
C^\infty_{0,\sigma}(\mathbb R^2))
\)
be a smooth, compactly supported, divergence-free and \(T\)-periodic
test function. The weak formulation of \eqref{eq:Vn-equation} is
\begin{equation}
	\label{eq:Vn-weak-formulation}
	\begin{aligned}
		&-\int_0^T\int_{\mathbb R^2}
		V_n\cdot\partial_s\Phi\,dy\,ds
		+
		\int_0^T\int_{\mathbb R^2}
		\nabla V_n:\nabla\Phi\,dy\,ds  \\
		&
		-
		\frac12\int_0^T\int_{\mathbb R^2}
		V_n\cdot\Phi\,dy\,ds
		-
		\frac12\int_0^T\int_{\mathbb R^2}
		(y\cdot\nabla V_n)\cdot\Phi\,dy\,ds  \\
		&
		-
		\int_0^T\int_{\mathbb R^2}
		\bigl(
		(U_{0,n}+V_n)
		  \otimes  
		(U_{0,n}+V_n)
		\bigr):\nabla\Phi\,dy\,ds
		=0.
	\end{aligned}
\end{equation}
Let \(K\subset\mathbb R^2\) be a compact set containing the spatial
support of \(\Phi\). By
\eqref{eq:linear-profile-L4-convergence} and
\eqref{eq:Vn-local-strong-L2},
\[
U_{0,n}+V_n
\longrightarrow
U_0+V
\quad\text{strongly in }
L^2((0,T)\times K).
\]
Consequently,
\begin{equation*}
	\begin{aligned}
		&(U_{0,n}+V_n)  \otimes  (U_{0,n}+V_n)
\longrightarrow
		(U_0+V)  \otimes  (U_0+V)
		\quad\text{strongly in }
		L^1((0,T)\times K).
	\end{aligned}
\end{equation*}
Passing to the limit in \eqref{eq:Vn-weak-formulation}, we obtain
\begin{equation*}
	\label{eq:V-limit-weak-formulation}
	\begin{aligned}
		&-\int_0^T\int_{\mathbb R^2}
		V\cdot\partial_s\Phi\,dy\,ds
		+
		\int_0^T\int_{\mathbb R^2}
		\nabla V:\nabla\Phi\,dy\,ds  \\
		&\quad
		-
		\frac12\int_0^T\int_{\mathbb R^2}
		V\cdot\Phi\,dy\,ds
		-
		\frac12\int_0^T\int_{\mathbb R^2}
		(y\cdot\nabla V)\cdot\Phi\,dy\,ds  \\
		&\quad
		-
		\int_0^T\int_{\mathbb R^2}
		\bigl(
		(U_0+V)  \otimes  (U_0+V)
		\bigr):\nabla\Phi\,dy\,ds
		=0.
	\end{aligned}
\end{equation*}
Thus \(V\) is a \(T\)-periodic weak solution of
\eqref{eq:difference}.

\subsection{Continuity in time and the initial condition}\label{sec:intialdata}

To verify the initial condition in physical variables, we  show
that the limiting profile is  continuous in \(L^2\).

\begin{lemma}
	The limiting profile \(V\) has a representative satisfying
	\begin{equation}
		\label{eq:V-continuous-L2}
		V\in C(\mathbb T_T;L^2(\mathbb R^2)).
	\end{equation}
	Consequently,
	\begin{equation}
		\label{eq:V-uniform-tail}
		\lim_{R\to\infty}
		\sup_{0\leq s\leq T}
		\int_{|y|>R}|V(y,s)|^2\,dy
		=0.
	\end{equation}
\end{lemma}

\begin{proof}
	For \(t\in[1,\lambda^2]\), define
	\begin{equation*}
		v(x,t)
		=
		\frac1{\sqrt t}
		V\left(\frac{x}{\sqrt t}, \log  t\right),
	\end{equation*}
	and set
	\[
	a(x,t)=e^{t\Delta}u_0(x).
	\]
	The change of variables between \((y,s)\) and \((x,t)\), together with
	\eqref{eq:V-limit-energy}, gives
	\[
	v\in
	L^\infty(1,\lambda^2;L^2(\mathbb R^2))
	\cap
	L^2(1,\lambda^2;H^1(\mathbb R^2)).
	\]
	Moreover, the two-dimensional Ladyzhenskaya inequality gives
	$v\in L^4((1,\lambda^2)\times\mathbb R^2)$.
	The linear field \(a\) satisfies that
	$
	a\in L^\infty(1,\lambda^2;L^4(\mathbb R^2))$.
	We note \(v\) satisfies 
	\begin{equation*}
		\partial_t v-\Delta v
		+
		\mathbb P\nabla\cdot
		\left((a+v)  \otimes  (a+v)\right)
		=0
	\end{equation*}
 Since
	\[
	(a+v)  \otimes  (a+v)
	\in
	L^2(1,\lambda^2;L^2(\mathbb R^2)),
	\]
	we have
	$\partial_tv
	\in
	L^2(1,\lambda^2;H^{-1}(\mathbb R^2))$.
	The Lions-Magenes lemma therefore yields
	$v\in C([1,\lambda^2];L^2(\mathbb R^2))$.

	For \(s\in[0,T]\),
	\[
	V(y,s)
	=
	e^{s/2}v(e^{s/2}y,e^s).
	\]
	The dilation
	\[
	f(x)\longmapsto e^{s/2}f(e^{s/2}x)
	\]
	is unitary on \(L^2(\mathbb R^2)\) and depends strongly continuously
	on \(s\). Hence
$V\in C([0,T];L^2(\mathbb R^2))$.
	Since \(V\) is \(T\)-periodic in the sense of distributions, its
	continuous representative satisfies
	\[
	V(\cdot,0)=V(\cdot,T)
	\quad\text{in }L^2(\mathbb R^2).
	\]
	This proves \eqref{eq:V-continuous-L2}.
	
	It remains to prove the uniform tail estimate. The set
$\mathcal K
	=
	\{V(\cdot,s):0\leq s\leq T\}$
 being the continuous image of the
	compact interval \([0,T]\), 
	is compact in \(L^2(\mathbb R^2)\). Compact subsets of \(L^2(\mathbb R^2)\)
	have uniformly  spatial tails. Thus, it holds that
	\[
	\sup_{0\leq s\leq T}
	\int_{|y|>R}|V(y,s)|^2\,dy
	\longrightarrow0\quad 	\text{as } R\to\infty.
	\]
This finishes the proof of the lemma.
\end{proof}

We are now ready to prove that $u$ is a finite energy perturbed solution.  It is easy to see
\[
u-e^{t\Delta}u_0
\in
L^2(\varepsilon,T_1;H^1(\mathbb R^2)).
\]
Finally, let
\(
K\Subset\mathbb R^2\setminus\{0\} 
\text{ and }
d_K=\operatorname{dist}(K,0)>0.
\)
One has
\begin{align*}
	\|u(t)-e^{t\Delta}u_0\|_{L^2(K)}^2
	&=
	\int_K
	\frac1t
	\left|
	V\left(\frac{x}{\sqrt t}, \log  t\right)
	\right|^2\,dx \notag\\
	&=
	\int_{K/\sqrt t}
	|V(y, \log  t)|^2\,dy \notag\\
	&\leq
	\sup_{0\leq s\leq T}
	\int_{|y|\geq d_K/\sqrt t}
	|V(y,s)|^2\,dy.
\end{align*}
By \eqref{eq:V-uniform-tail}, the right-hand side converges to zero
as \(t\to0\). Therefore,
\begin{equation*}
	\lim_{t\to0}
	\|u(t)-e^{t\Delta}u_0\|_{L^2(K)}
	=0.
\end{equation*}
Since \(K\subset\subset\mathbb R^2\setminus\{0\}\) is arbitrary, \(u\) is an
energy perturbed solution in the sense of
Definition~\ref{def:energy persol}.

\appendix
\section{Technical proofs}\label{sec:App}
In this appendix, we give the proof of Lemmas \ref{lem:Estimates for modified linear part} and \ref{lem:weighted-profile-Stokes}.
\begin{proof}[Proof of Lemma \ref{lem:Estimates for modified linear part}]
		The smoothness of $U_1$ and $\nabla \cdot U_1=0$ is clear from the construction. We  focus on proving \eqref{eq:estw}-\eqref{eq:smallness}.
		Estimates \eqref{eq:estnalbav1} and \eqref{eq:estv1} are easy to see using Lemma \ref{lem:Linftyestforheat} and \eqref{eq:estw}-\eqref{eq:estsw}.
		We  first prove \eqref{eq:estw}, and we only  consider the case for $|y|\geq 3R_0$, as the boundedness and smoothness of $w$ inside $|y|\leq 3R_0$ are implied directly by the boundedness  of $U_0$ in Lemma \ref{lem:Linftyestforheat}.
		Since 
		$U_0 \cdot \nabla  \eta_{R_0}$ is smooth and has compact support in $B_{2R_0}\setminus B_{R_0}$, we have
		\begin{equation*}
\int_{\mathbb{R}^2} U_0 \cdot \nabla  \eta_{R_0} = \int_{B_{2R_0}\setminus B_{R_0}} \textrm{ div } (\eta_{R_0} U_0) =  \int_{\partial B_{2R_0}}  U_0\cdot n d\sigma=0.
		\end{equation*}
If $|y|\geq 3R_0$, then for  $|z|\leq 2R_0$, one has
\[
\left|
\frac{y-z}{|y-z|^2}
-
\frac{y}{|y|^2}
\right|
\leq
C\frac{|z|}{|y|^2}
\leq
C\frac{R_0}{|y|^2}.
\]
Hence, using Lemma \ref{lem:Linftyestforheat}, it holds that
	\begin{equation*}
				\begin{aligned}
						|w(y,s)| \leq 
		 \int_{B_{2R_0}}  \frac{CR_0}{|y|^2}|U_0 \cdot \nabla  \eta_{R_0}|dz \leq \frac{C(A, \lambda, R_0)}{|y|^2}.
					\end{aligned}
			\end{equation*}
		This, together with the boundedness of $w$ inside $B_{3R_0}$, shows that 
		\begin{equation*}
			|  w(y,s)|\leq \frac{C(A,  \lambda, R_0)}{1+ |y|^2}.
		\end{equation*}
		Taking $\sup$ in $s\in[0,T]$, we have 
			\begin{equation*}
			\sup_{0\leq s\leq T} |  w(y,s)|\leq \frac{C(A,  \lambda, R_0)}{1+ |y|^2}.
		\end{equation*}
		The proof for \[
        \sup_{0\leq s\leq T} |  \nabla^k w(y,s)|\leq \frac{C(A,\lambda, k, R_0)}{1+|y|^{k+2}},\quad  k\geq 1,
        \]
		is similar to the above calculation by taking the derivatives into the kernel and using the property $\int_{\mathbb{R}^2} U_0 \cdot \nabla  \eta_{R_0} = 0$.  We omit details.
		
		We next prove the corresponding estimate for $\partial_s w$.
		Differentiating \eqref{eq:corr} with respect to $s$, we get
		\begin{equation*}
			\partial_s w(y,s)
			=
			-\frac{1}{2\pi}
			\int_{\mathbb R^2}
			\frac{y-z}{|y-z|^2}
			\, \partial_s U_0(z,s)\cdot\nabla\eta_{R_0}(z)\,dz .
		\end{equation*}
		Here the differentiation under the integral sign is justified since
		$\nabla\eta_{R_0}$ is compactly supported in
		$B_{2R_0}\setminus B_{R_0}$ and $U_0$ is smooth for all $s$.
		Since $\nabla\cdot U_0=0$, we also have
		$\nabla\cdot \partial_s U_0=0$.
 Hence the flux of
		$\partial_s U_0$ through every circle vanishes
		\[
		\int_{\partial B_r} \partial_s U_0\cdot n\,d\sigma
		=
		\partial_s\int_{\partial B_r} U_0\cdot n\,d\sigma
		=
		\partial_s\int_{B_r}\nabla\cdot U_0\,dy
		=0 .
		\]
		Therefore,
		\begin{equation}\label{eq:cancellation-sw}
			\int_{\mathbb R^2}
			\partial_s U_0(z,s)\cdot\nabla\eta_{R_0}(z)\,dz
			=
			\int_{B_{2R_0}\setminus B_{R_0}}
			\nabla\cdot\bigl(\eta_{R_0}\partial_s U_0\bigr)\,dz
			=
			\int_{\partial B_{2R_0}}\partial_s U_0\cdot n\,d\sigma
			=
			0 .
		\end{equation}
		We also note that, by \eqref{eq:equforheateq},
		\[
		\partial_s U_0
		=
		\Delta U_0+\frac12 U_0+\frac12 y\cdot\nabla U_0 .
		\]
		Since $\nabla\eta_{R_0}$ is supported in the fixed annulus
		$B_{2R_0}\setminus B_{R_0}$, Lemma \ref{lem:Linftyestforheat} implies
		\begin{equation}\label{eq:source-sw-bound}
			\sup_{0\leq s\leq T}
			\left\|
			\partial_s U_0(\cdot,s)\cdot\nabla\eta_{R_0}
			\right\|_{L^1(\mathbb R^2)}
			\leq C(A,\lambda,R_0).
		\end{equation}
	For $|y|\ge 3R_0$, using \eqref{eq:cancellation-sw}, we may write
	\[
	\partial_s w(y,s)
	=
	-\frac{1}{2\pi}
	\int_{B_{2R_0}}
	\left(
	\frac{y-z}{|y-z|^2}
	-
	\frac{y}{|y|^2}
	\right)
	\partial_s U_0(z,s)\cdot\nabla\eta_{R_0}(z)\,dz .
	\]
	Since $|z|\leq 2R_0$ and $|y|\ge3R_0$, one has
	\[
	\left|
	\frac{y-z}{|y-z|^2}
	-
	\frac{y}{|y|^2}
	\right|
	\leq
	C\frac{|z|}{|y|^2}
	\leq
	C\frac{R_0}{|y|^2}.
	\]
	Together with \eqref{eq:source-sw-bound}, this yields
	\[
	|\partial_s w(y,s)|
	\leq
	\frac{C(A,\lambda,R_0)}{|y|^2},
	\quad |y|\ge 3R_0.
	\]
	For higher derivatives, the proof	is similar to the above calculations.
	Combining the interior bound with the far-field estimate proves \eqref{eq:estsw}.

		To prove \eqref{eq:smallness}, we
		first note  due to Lemma \ref{lem:Linftyestforheat} and the definition of $\eta_{R_0}$,
		 for any $r>2$ one has
		\begin{equation*}
			\sup_{0\leq s\leq T}\| \eta_{R_0}U_0\|_{W^{2,r}(\mathbb{R}^2)} \to 0 \textrm{ as } R_0\to +\infty.
		\end{equation*}
		Using the integral formula for $w$ and the Calder\'{o}n–Zygmund estimates, we have for any $r>2$ that
		\begin{equation*}
					\sup_{0\leq s\leq T}\|w\|_{W^{2,r}(\mathbb{R}^2)}\leq c(r) 		\sup_{0\leq s\leq T} \| \eta_{R_0}U_0\|_{W^{2,r}(\mathbb{R}^2)}.
		\end{equation*}
		It also follows from Morrey's inequality that for any $r>2$
		\begin{equation*}
			\|w\|_{C^{1,\alpha}(\mathbb{R}^2)}\leq c(r) \|w\|_{W^{2,r}(\mathbb{R}^2)},
		\end{equation*}
		where $\alpha = 1-\frac{2}{r}$.
		Hence, we have
		\begin{equation*}
					\sup_{0\leq s\leq T} \|w\|_{C^{1,\alpha}(\mathbb{R}^2)}\leq 	c(r)		\sup_{0\leq s\leq T} \| \eta_{R_0}U_0\|_{W^{2,r}(\mathbb{R}^2)} \to 0 \textrm{ as } R_0\to +\infty.
		\end{equation*}
		Choose $r=3$ in the above, 
		for any $\epsilon>0$, we can then choose $R_0 = R_0(A, \lambda,\epsilon)$ large enough such that
		\begin{equation*}
				\sup_{0\leq s\leq T}	\|w\|_{C^{1,\frac{1}{3}}(\mathbb{R}^2)}\leq C 		\sup_{0\leq s\leq T} \| \eta_{R_0}U_0\|_{W^{2,3}(\mathbb{R}^2)}\leq  \frac{\epsilon}{2},
		\end{equation*}
	and also
		\begin{equation*}
					\sup_{0\leq s\leq T} \|\eta_{R_0}	 U_0\|_{C^{1}(\mathbb{R}^2)} \leq\frac{\epsilon}{2}.
		\end{equation*}
		Combining the above two inequalities, we obtain  \eqref{eq:smallness}. 
		This finishes the  proof of the lemma.
	\end{proof}

We next prove Lemma \ref{lem:weighted-profile-Stokes}.
\begin{proof}[Proof of Lemma \ref{lem:weighted-profile-Stokes}]
	Let $\mathcal O(x,t)$ be the kernel of
	$e^{t\Delta}\mathbb P\nabla\cdot$. In two dimensions, one has
	\begin{equation*}
		|\mathcal O(x,t)|
		\leq C(|x|+\sqrt t)^{-3}\quad \text{and}\quad 
		\mathcal O(x,t) = t^{-\frac32} \mathcal{O}\left(\frac{x}{\sqrt{t}}, 1\right).
	\end{equation*}
	Set
	\[
	q=\frac{30}{23}.
	\]
	For $0<a\leq1$, one has
	\begin{equation*}
		\begin{aligned}
			&\left\|
			(1+|x|^2)^{\beta/2}\mathcal O(x,a)
			\right\|_{L^q(\mathbb R^2)} = \left\|
			(1+|x|^2)^{\beta/2}a^{-\frac{3}{2}}\mathcal O(x/\sqrt{a},1)
			\right\|_{L^q(\mathbb R^2)}  \\
			\leq	&
			Ca^{-\frac{3}{2}+1/q}
			\left\|
			 (1+|z|^2)^{\beta/2}\mathcal O(z,1)
			\right\|_{L^q(\mathbb R^2)}\leq C(\beta)a^{-\frac{11}{15}},
		\end{aligned}
	\end{equation*}
	where we have used  that the   norm in the second to last  inequality is finite
	for $0\leq\beta<1$.
	
	For $0<\tau<1$, we let
	\[
	G_\tau(y,s)
	=
	\tau^{-1}F\left(\frac y{\sqrt \tau},s+ \log  \tau\right).
	\]
	Using
	\[
	(1+|y|^2)^{\beta/2}
	\leq
	C_\beta
	(1+|y-z|^2)^{\beta/2} (1+|z|^2)^{\beta/2}
	\]
	and Young's convolution inequality, we obtain
	\begin{equation}
		\label{eq:weighted-spatial-convolution-LS}
		\begin{aligned}
			&\left\|
			(1+|y|^2)^{\beta/2}
			e^{(1-\tau)\Delta}\mathbb P\nabla\cdot G_\tau(\cdot,s)
			\right\|_{L^6(\mathbb R^2)}\\
			\leq&
			C\left\|		(1+|\cdot|^2)^{\beta/2} \mathcal{O}(\cdot,1-\tau)\right\|_{L^{\frac{30}{23}}(\mathbb{R}^2)}
			\left\|
			(1+|\cdot|^2)^{\beta/2}G_\tau(\cdot,s)
			\right\|_{L^{\frac{5}{2}}(\mathbb R^2)}\\
			\leq&
			C(1-\tau)^{-\frac{11}{15}}
			\left\|
			(1+|\cdot|^2)^{\beta/2}G_\tau(\cdot,s)
			\right\|_{L^{\frac{5}{2}}(\mathbb R^2)}\\
			\leq &		C(1-\tau)^{-\frac{11}{15}} \tau^{-\frac{3}{5}}
			\left\|
			 (1+|z|^2)^{\beta/2}F(z,s+ \log  \tau)
			\right\|_{L^{\frac{5}{2}}(\mathbb R^2)}.
		\end{aligned}
	\end{equation}
	Set
	\[
	f(s)
	=
	\left\|
	(1+|y|^2)^{\beta/2}F(\cdot,s)
	\right\|_{L^{\frac{5}{2}}(\mathbb R^2)}.
	\]
	It follows from \eqref{eq:ZF-definition} and
	\eqref{eq:weighted-spatial-convolution-LS},
	one has
	\begin{equation}
		\label{eq:time-convolution-before-substitution}
		\begin{aligned}
			\left\|
			(1+|y|^2)^{\beta/2}Z_F(\cdot,s)
			\right\|_{L^6(\mathbb R^2)}
			&\leq
			C\int_0^1
			(1-\tau)^{-\frac{11}{15}}\tau^{-\frac{3}{5}}f( s +  \log  \tau)\,d\tau\\
			&=C\int_0^\infty
			(1-e^{-\theta})^{-\frac{11}{15}}e^{-2\theta/5} f(s-\theta)\,d\theta\\
		\end{aligned}
	\end{equation}
	Define
	\begin{equation*}
		h(\theta)
		:=
		(1-e^{-\theta})^{-\frac{11}{15}}
		e^{-\frac{2\theta}{5}},
		\quad \theta>0.
	\end{equation*}
	It is easy to see
	\(
	h\in L^{30/23}(0,\infty)
	\).
	To use the periodicity of \(f\), one writes
	\begin{align*}
		\int_0^\infty h(\theta)f(s-\theta)\,d\theta
		&=
		\sum_{k=0}^\infty
		\int_{kT}^{(k+1)T}
		h(\theta)f(s-\theta)\,d\theta                                    \\
		&=
		\sum_{k=0}^\infty
		\int_0^T
		h(\theta +kT)f(s-\theta)\,d\theta=	\int_0^T h_T(\theta)f(s-\theta)\,d\theta,
	\end{align*}
	where
	\begin{equation*}
		h_T(\theta)
		:=
		\sum_{k=0}^\infty h(\theta+kT),
		\quad 0\leq \theta<T.
	\end{equation*}
	We then have
	\begin{equation}\label{eq:time-convolution-before-substitution1}
		\left\|
		(1+|y|^2)^{\frac{\beta}{2}}
		Z_F(\cdot,s)
		\right\|_{L^6(\mathbb R^2)}
		\leq
		C\int_0^T h_T(\theta)f(s-\theta)\,d\theta .
	\end{equation}
	We claim that
	\(
	h_T\in L^{30/23}(0,T)
	\).
	Indeed,  we can write
	\[
	h_T(\tau)
	=
	h(\tau)
	+
	\sum_{k=1}^{\infty}h(\tau+kT).
	\]
	For \(k\geq1\) and \(0<\tau<T\), one has
	\[
	1-e^{-(\tau+kT)}
	\geq
	1-e^{-T}.
	\]
	Consequently,
	\begin{align*}
		\sum_{k=1}^{\infty}h(\tau+kT)
		&\leq
		(1-e^{-T})^{-\frac{11}{15}}
		\sum_{k=1}^{\infty}
		e^{-\frac25(\tau+kT)}\\
		&=
		(1-e^{-T})^{-\frac{11}{15}}
		e^{-\frac{2\tau}{5}}
		\frac{e^{-\frac{2T}{5}}}
		{1-e^{-\frac{2T}{5}}}
		\leq C(T).
	\end{align*}
	Hence, one has
	\(
	h_T(\tau)\leq h(\tau)+C(T)
	\), and it is obviously that $h(\cdot)\in L^{30/23}(0,T)$. This proves the claim.
	We now apply Young's convolution inequality to \eqref{eq:time-convolution-before-substitution1} on the  torus
	\(\mathbb{T}_{T}\) to obtain
	\begin{equation*}
		\left\|
		(1+|y|^2)^{\frac{\beta}{2}}Z_F
		\right\|_{L^6(\mathbb R^2\times(0,T))}
		\leq
		C
		\|h_T\|_{L^{30/23}(0,T)}
		\|f\|_{L^{\frac{5}{2}}(0,T)}                                                
		=
		C
		\left\|
		(1+|y|^2)^{\frac{\beta}{2}}F
		\right\|_{L^{\frac{5}{2}}(\mathbb R^2\times(0,T))}.
	\end{equation*}
	This proves the desired weighted Stokes estimate
	\eqref{eq:weighted-profile-Stokes-estimate}.
	
	We next prove \eqref{eq:local-profile-Stokes-estimate}. On
	$B_{2R}\times(0,T)$, equation \eqref{eq:forced-linear-profile} may be
	rewritten as
	\begin{equation*}
		\begin{aligned}
			\partial_sZ_F-\Delta Z_F+\nabla P_F
			=
			\nabla\cdot\left(-F+\frac12Z_F  \otimes   y\right)
			-\frac12Z_F.
		\end{aligned}
	\end{equation*}
	By the local \(L^p\)-regularity theory for the nonstationary Stokes system in divergence form, together with the standard weak \(L^p\)-estimate for the time derivative, we obtain
	\begin{align*}
		&\|Z_F\|_{L^{\frac{5}{2}}(0,T;W^{1,\frac{5}{2}}(B_R))}
		+
		\|\partial_sZ_F\|_{L^{\frac{5}{2}}(0,T;W^{-1,\frac{5}{2}}(B_R))}
		\leq
		C_R\left(
		\|Z_F\|_{L^{\frac{5}{2}}(B_{2R}\times(0,T))}
		+
		\|F\|_{L^{\frac{5}{2}}(B_{2R}\times(0,T))}
		\right).
	\end{align*}
	This, together with
	\eqref{eq:weighted-profile-Stokes-estimate} and the H\"older's inequality, proves
	\eqref{eq:local-profile-Stokes-estimate}.
\end{proof}

\medskip 

{\bf Acknowledgement.}
		The research of Gui is supported by  NSFC Key Program (Grant No. 12531010),  University of Macau research grants CPG2024-00016-FST, CPG2025-00032-FST, CPG2026-00027-FST, SRG2023-00011-FST, MYRG-GRG2023-00139-FST-UMDF, UMDF Professorial Fellowship of Mathematics, Macao SAR FDCT 0003/2023/RIA1 and  Macao SAR FDCT 0024/2023/RIB1.   The research of  Xie is partially supported by  NSFC grants 12571238 and 12426203.
        
\bibliographystyle{abbrv}

\end{document}